\documentclass{amsart}
\usepackage{epsfig}

\newtheorem{theorem}{Theorem}[section]
\newtheorem{lemma}[theorem]{Lemma}
\newtheorem{corollary}[theorem]{Corollary}
\newtheorem{proposition}[theorem]{Proposition}

\theoremstyle{definition}
\newtheorem{definition}[theorem]{Definition}

\newtheorem{note}[theorem]{Note}

\theoremstyle{remark}

\begin{document}

\title[Zagier polynomials. Part I.]{The Zagier modification of 
Bernoulli numbers and a polynomial 
extension. Part I.}

\author[A. Dixit]{Atul Dixit}
\address{Department of Mathematics,
Tulane University, New Orleans, LA 70118}
\email{adixit@tulane.edu}

\author[V. H. Moll]{Victor H. Moll}
\address{Department of Mathematics,
Tulane University, New Orleans, LA 70118}
\email{vhm@tulane.edu}

\author[C. Vignat]{Christophe Vignat}
\address{Department of Mathematics,
Tulane University, New Orleans, LA 70118 and 
L.S.S. Supelec, Universite d'Orsay, France}
\email{vignat@tulane.edu}

%    General info

\date{\today}


\begin{abstract}
The modified
\begin{equation*}
B_{n}^{*} = \sum_{r=0}^{n} \binom{n+r}{2r} \frac{B_{r}}{n+r}, \quad n > 0
\end{equation*}
\noindent
introduced by D. Zagier in $1998$ are extended to the polynomial 
case by replacing $B_{r}$ by the Bernoulli polynomials $B_{r}(x)$. Properties 
of these new polynomials are established using the umbral 
method as well as classical techniques. The values of $x$ that yield periodic 
subsequences $B_{2n+1}^{*}(x)$ are classified. The strange $6$-periodicity of 
$B_{2n+1}^{*}$, established by Zagier, is 
explained by exhibiting a decomposition of this sequence as 
the sum of two parts with periods $2$ and $3$, respectively. Similar 
results for modifications of Euler numbers are stated.
\end{abstract}

\maketitle

\newcommand{\ba}{\begin{eqnarray}}
\newcommand{\ea}{\end{eqnarray}}
\newcommand{\ift}{\int_{0}^{\infty}}
\newcommand{\nn}{\nonumber}
\newcommand{\no}{\noindent}
\newcommand{\A}{\mathfrak{A}}
\newcommand{\B}{\mathfrak{B}}
\newcommand{\C}{\mathfrak{C}}
\newcommand{\D}{\mathfrak{D}}
\newcommand{\E}{\mathfrak{E}}
\newcommand{\pe}{\mathfrak{P}}
\newcommand{\Q}{\mathfrak{Q}}
\newcommand{\U}{\mathfrak{U}}
\newcommand{\lf}{\left\lfloor}
\newcommand{\rf}{\right\rfloor}
\newcommand{\vm}{{}}
\newcommand{\realpart}{\mathop{\rm Re}\nolimits}
\newcommand{\imagpart}{\mathop{\rm Im}\nolimits}

\newtheorem{Definition}{\bf Definition}[section]
\newtheorem{Thm}[Definition]{\bf Theorem} 
\newtheorem{Example}[Definition]{\bf Example} 
\newtheorem{Lem}[Definition]{\bf Lemma} 
\newtheorem{Cor}[Definition]{\bf Corollary} 
\newtheorem{Prop}[Definition]{\bf Proposition} 
\numberwithin{equation}{section}

\section{Introduction}
\label{sec-intro} 
\setcounter{equation}{0}

The Bernoulli numbers $B_{n}$, defined by the generating function 
\begin{equation}
\frac{t}{e^{t}-1} = \sum_{n=0}^{\infty} B_{n} \frac{t^{n}}{n!}
\label{bn-def}
\vm
\end{equation}
\noindent 
are rational numbers with $B_{2n+1} = 0$ for 
$n \geq 1$ and $B_{1} = - \tfrac{1}{2}$. The sequence $\{ B_{n} \}$ 
has remarkable properties and it appears in a variety of mathematical problems. 
Examples of such include the fact that the Riemann zeta function 
\begin{equation}
\zeta(s) = \sum_{n=1}^{\infty} \frac{1}{n^{s}} \vm
\end{equation}
\noindent 
evaluated at an even positive integer $s= 2n$ is a rational 
multiple of $\pi^{2n}$, the factor being 
\begin{equation}
\frac{\zeta(2n)}{\pi^{2n}}  = \frac{2^{2n-1}}{(2n)!} (-1)^{n-1} B_{2n}.
\vm
\end{equation}
\noindent
Their denominators are completely 
determined by the von Staudt-Clausen theorem: the denominator of $B_{2n}$ is 
the product of all primes $p$ such that $p-1$ divides $2n$ (see 
\cite{macmillan-2011a} for an elementary proof).   It is often the numerators 
of $B_{2n}$ that are the objects of interest. It is a remarkable mystery that 
there is no elementary formula associated to them. These numerators appear in 
connection to Fermat's last theorem (see \cite{ribenboim-1999a}) and also
in relation to the group of smooth structures on $n$-spheres 
(see \cite{kervaire-1963a}, page 530 and \cite{levine-1983a} for details). 

D. Zagier \cite{zagier-1998a} introduced the modified Bernoulli numbers 
\begin{equation}
B_{n}^{*} = \sum_{r=0}^{n} \binom{n+r}{2r} \frac{B_{r}}{n+r}, \quad n > 0
\label{zagier-mod}
\end{equation}
\noindent
and established the following \textit{amusing variant} of $B_{2n+1}=0$:

\smallskip

\begin{theorem}
\label{zagier-per6}
The sequence $B_{2n+1}^{*}$ is periodic of period $6$ with values 
$$\{ \tfrac{3}{4}, - \tfrac{1}{4}, \, - \tfrac{1}{4}, \, 
\tfrac{1}{4}, \, \tfrac{1}{4}, \, - \tfrac{3}{4}  \}. \vm $$
\end{theorem}

One of the goals of this work is to extend this result to the polynomial
\begin{equation}
\label{zag-polydef}
B_{n}^{*}(x) = \sum_{r=0}^{n} \binom{n+r}{2r} \frac{B_{r}(x)}{n+r}, 
\quad n > 0
\end{equation}
\noindent
which is defined here as the \textit{Zagier polynomial}. Here 
$B_{r}(x)$ is the classical Bernoulli polynomial defined by the 
generating function 
\begin{equation}
\frac{te^{xt}}{e^{t}-1} = \sum_{n=0}^{\infty} B_{n}(x) \frac{t^{n}}{n!}.
\label{gf-berpoly}
\end{equation}
\noindent
The objective of the paper is to produce analogues of 
standard results for $B_{n}(x)$ for the Zagier polynomials
$B_{n}^{*}(x)$. For example, a generating function for these polynomials 
appears in Theorem \ref{gf-bstar-poly} as 
\begin{equation}
\sum_{n=1}^{\infty} B_{n}^{*}(x)z^{n} = - \frac{1}{2} \log z - 
\frac{1}{2} \psi \left( z + 1/z - 1 - x \right)
\label{gf-zagier1}
\end{equation}
\noindent
where $\psi(x) = \Gamma'(x)/\Gamma(x)$ is the digamma function. The generating 
function \eqref{gf-zagier1} really corresponds to the less elementary 
expression 
\begin{equation}
\sum_{n=0}^{\infty} B_{n}(x) z^{n} = \frac{1}{z} \zeta(2, 1/z-x+1).
\label{gf-bernoulli1}
\end{equation}
\noindent 
Here $\zeta(s,a)$ is the Hurwitz zeta function, defined by 
\begin{equation}
\zeta(s,a) = \sum_{n=0}^{\infty} \frac{1}{(n+a)^{s}}.
\end{equation}
\noindent
A derivation of \eqref{gf-bernoulli1} is given in Section \ref{gf-zagierpoly}.
The next example corresponds to the derivative rule 
$B_{n}'(x) = nB_{n-1}(x)$ for the Bernoulli polynomials. It 
appears in Theorem \ref{der-mod-zag} as 
\begin{equation*}
\frac{d}{dx} B_{n}^{*}(x) = \sum_{j=1}^{\lf \frac{n}{2} \rf} 
(2j-1) B_{2j-1}^{*}(x) \quad \text{ for } n \text{ even}
\end{equation*}
\noindent
and 
\begin{equation*}
\frac{d}{dx} B_{n}^{*}(x) = \frac{1}{2} + \sum_{j=1}^{\lf \frac{n}{2} \rf} 
2j B_{2j}^{*}(x)  \quad \text{ for } n \text{ odd}.
\end{equation*}
\noindent
Finally, the analogue of the symmetry 
relation $B_{n}(1-x) = (-1)^{n}B_{n}(x)$ is established as 
\begin{equation}
B_{n}^{*}(-x-3) = (-1)^{n}B_{n}^{*}(x).
\end{equation}
\noindent
This is the content of Theorem \ref{thm-symm}.

The original motivating factor for this work was to extend Theorem 
\ref{zagier-per6} to other values of $B_{2n+1}^{*}(x)$. The main result 
presented here is a classification of the values $x \in \mathbb{R}$ for which 
$B_{2n+1}^{*}(x)$ is a periodic sequence, Zagier's case being 
$x=0$. 

\begin{theorem}
\label{zag-thm1a}
Assume $\{ B_{2n+1}^{*}(x) \}$ is a periodic sequence. Then $x \in 
\{-3, \, -2, \, -1, \, 0 \}$ or $x = - \tfrac{3}{2}$ where 
$B_{2n+1}^{*} \left( - \tfrac{3}{2} \right) = 0.$ 
\end{theorem}

In the case of even degree, the natural result is expressed in terms of the 
difference $A_{2n}^{*}(x) = B_{2n}^{*}(x) - B_{2n}^{*}(-1)$. 

\begin{theorem}
\label{zag-thm1b}
Assume $\{ A_{2n}^{*}(x) \}$  is a periodic sequence. Then $x \in \{ -1, \, 
0, \, 1, \, 2 \}$. The period is $3$ for $x = -1$ and $x=2$ 
while  $A_{2n}^{*}(x)$
vanishes identically for $x=0$ and $x=1$.
\end{theorem}

An outline of the paper is given next. 
Section \ref{sec-umbral} contains a  basic introduction to the umbral 
calculus with a special emphasis on the operational rules for the Bernoulli 
umbra. Section \ref{modified-ber} gives the generating function of the 
modified Bernoulli numbers $B_{n}^{*}$ and this is used to give a proof 
of the $6$-periodicity of $B_{2n+1}^{*}$. An inversion formula expressing 
$B_{n}$ in terms of $B_{n}^{*}$ is given in Section \ref{inv-modber1}. The
proof extends to the polynomial case.  The generating function for the 
Zagier polynomial $B_{n}^{*}(x)$ is established in Section \ref{gf-zagierpoly}
and an introduction to the arithmetic properties of special values of these 
polynomials appears in Section \ref{sec-arithmetic}. Expressions for the 
derivatives of the Zagier polynomials are given in Section 
\ref{sec-derivatives}. Some binomials sums employed in the proof of these 
results are given in Section \ref{sec-binsums}. The basic properties of 
Chebyshev polynomials are reviewed in Section \ref{sec-chebyshev} and used 
in Section \ref{sec-representation} to give a representation of the 
Zagier polynomials in terms of Bernoulli and Chebyshev polynomials and also 
to prove a symmetry property of $B_{n}^{*}(x)$ in Section 
\ref{mod-ber-poly-reflection}. Section 
\ref{sec-periodic} contains 
one of the main results: the classification of all periodic sequences of the 
form $B_{2n+1}^{*}(x)$. This result extends the original theorem of D. Zagier
on the $6$-periodicity of $B_{2n+1}^{*}$. Several additional properties  of 
the Zagier polynomials are stated in Section \ref{sec-additional}. The 
results discussed in the present paper can be extended without
difficulty to the case of Euler polynomials. These extensions 
are stated in  Section \ref{sec-euler} and used in Section \ref{sec-duplication}
to establish a duplication formula for Zagier polynomials.

\section{The umbral calculus}
\label{sec-umbral} 
\setcounter{equation}{0}

In the classical \textit{umbral calculus}, as introduced by J. Blissard 
\cite{blissard-1861a},  the terms $a_{n}$ of a sequence are
formally replaced by the powers ${\mathfrak{a}}^{n}$ of a new variable 
$\mathfrak{a}$, named the \textit{umbra} of the sequence $\{ a_{n} \}$. The 
original sequence is recovered by the \textit{evaluation map} 
\begin{equation}
{\rm{eval}} \left\{ \mathfrak{a}^{n} \right\} = a_{n}.
\end{equation}
\noindent
The introduction of an umbra for $\{ a_{n} \}$ requires a 
\textit{constitutive equation} that reflects the properties of the original 
sequence. These ideas are illustrated with the umbra $\B$ of the Bernoulli 
numbers $\{ B_{n} \}$.

An alternative approach to \eqref{bn-def}, as a definition for the 
Bernoulli numbers $B_{n}$,  is to use the equivalent recursion formula
\begin{equation}
\sum_{k=0}^{n-1} \binom{n}{k} B_{k} = 0, \quad \text{ for } n > 1,
\end{equation}
\noindent
complemented by the initial condition $B_{0}=1$. In terms of the Bernoulli 
umbra $\B$, this recursion is written as 
\begin{equation}
- \B  =   \B + 1. \label{ber-umbrae}
\end{equation}
\noindent
This is a  constitutive equation for the Bernoulli umbra and the 
numbers $B_{n}$ are then obtained via the evaluation map
\begin{equation}
{\rm{eval}} \{ \mathfrak{B}^{n} \} = B_{n}.
\end{equation}
\noindent
The umbral method is illustrated by computing the 
first few values of $B_{n}$, starting with 
the initial condition $B_{0} = 1$. The choice  $n=2$ in \eqref{ber-umbrae} 
gives 
\begin{equation}
\mathfrak{B}^{2} = (\mathfrak{B}+1)^{2} = \mathfrak{B}^{2} + 
2\mathfrak{B}^{1} + \mathfrak{B}^{0}.
\end{equation}
\noindent
The evaluation map then gives $B_{2} = B_{2} + 2B_{1} + B_{0}$ that 
simplifies to  $2B_{1}+B_{0} = 0$ and this yields 
$B_{1} = -1/2$. Similarly, $n=3$ gives 
\begin{equation}
\B^{3} = (\B+1)^{3} = \B^{3} + 3 \B^{2} + 3 \B^{1}  + \B^{0},
\end{equation}
\noindent
and the evaluation map produces 
$B_{3} + 3 B_{2} + 3 B_{1}  + B_{0} = B_{3}$ and 
$B_{2} = \tfrac{1}{6}$ is obtained. The 
reader will find more details about these ideas in \cite{gessel-2003a}.

The evaluation map of the Bernoulli umbra $\B$ may be defined at the level 
of generating functions by 
\begin{equation}
{\rm{eval}} \left\{ \exp( t \B) \right\} = \frac{t}{e^{t}-1}.
\label{bn-umbrae1}
\end{equation}
\noindent
Similarly, the 
extension of \eqref{bn-umbrae1} to the umbrae $\B(x)$ for the Bernoulli
polynomials in \eqref{gf-berpoly} is defined by 
\begin{equation}
{\rm{eval}} \left\{ \exp( t \B(x)) \right\} = \frac{te^{xt}}{e^{t}-1}.
\label{umbra-berpoly}
\end{equation}

It is a general statement about umbral calculus that the operation 
$\textit{eval}$ is linear. Moreover, expressions independent of the 
corresponding umbra are to be treated as constant with respect 
to $\textit{eval}$. Some further 
operational rules, particular for the Bernoulli umbra,  are stated next. 

\begin{lemma}
\label{cool-sum}
The relation 
\begin{equation}
{\rm{eval}} \{ \B(x) \} =  {\rm{eval}} \{x + \B \}.
\label{umbra-rel1}
\end{equation}
\noindent
holds.
\end{lemma}
\begin{proof}
This comes directly from
\begin{equation}
{\rm{eval}} 
\left\{ \exp(t \B(x)) \right\} = \frac{te^{xt}}{e^{t}-1} \text{ and }
{\rm{eval}} \left\{ \exp(t \B) \right\} = \frac{t}{e^{t}-1}.
\end{equation}
\end{proof}

\begin{lemma}
\label{taylor}
Let $P$ be a polynomial. Then 
\begin{equation}
{\rm{eval}} \left\{ P( x + \B + 1 ) \right \} = 
{\rm{eval}} \left\{ P( x + \B ) \right \} + P'(x).
\end{equation}
\end{lemma}
\begin{proof}
This is verified first for monomials using \eqref{ber-umbrae} and then extended 
by linearity to the polynomial case. 
\end{proof}

The next step is to present a procedure to evaluate nonlinear functions of 
the Bernoulli polynomials. This can be done using the umbral approach but 
we introduce here an equivalent probabilistic formalism that is easier to 
use in a variety of examples. These two approaches will be described further 
in \cite{dixit-2012b} where some of the results presented in 
\cite{yu-2012a} will be established using this formalism. 

The notation 
\begin{equation}
\mathbb{E} \left[ h(X) \right] = \int_{\mathbb{R}} h(x) \, f_{X}(x) \, dx
\label{int-expec}
\end{equation}
\noindent
is used here for the expectation operator based on the random variable $X$ with 
probability 
density $f_{X}$. The class of admissible functions $h$ is restricted by the 
existence of the integral \eqref{int-expec}. The equation \eqref{eval-98} 
shows that a probabilistic equivalent of the \textit{eval} operator of umbral 
calculus is the expectation operator with respect to the probability 
distribution \eqref{form-lb}.

\begin{theorem}
\label{density-B}
There exists a real valued 
random variable $L_{B}$ with probability density $f_{L_{B}}(x)$
such that, for all admissible functions $h$, 
\begin{equation}
{\rm{eval}}\{ h(\mathfrak{B}(x)) \} = 
\mathbb{E} \left[ h( x - 1/2 + i L_{B}) \right]
\label{eval-98}
\end{equation}
\noindent
where the expectation is defined in \eqref{int-expec}. The density 
of $L_{B}$ is given by 
\begin{equation}
f_{L_{B}}(x) = \frac{\pi}{2} {\rm{sech}}^{2}(\pi x), \quad \text{ for } 
x \in \mathbb{R}.
\label{form-lb}
\end{equation}
\end{theorem}
\begin{proof}
Put $x=0$ in the special case 
\begin{equation}
{\rm{eval}}\left\{\exp( t \mathfrak{B}(x)) \right\} = 
\mathbb{E} \left[ \exp \left( t( x - \tfrac{1}{2} + i L_{B}) \right)  \right]
\label{val-11}
\end{equation}
\noindent
of \eqref{eval-98} and use \eqref{umbra-berpoly} to produce 
\begin{equation}
\mathbb{E} \left[ \exp( it L_{B} )  \right] = 
\frac{t}{2}\text{csch} \left( \frac{t}{2} \right).
\label{eq-1}
\end{equation}
\noindent
Let $f_{L_{B}}(x)$ be the density of $L_{B}$ and write \eqref{eq-1} as 
\begin{equation}
\int_{-\infty}^{\infty} \cos(tu) f_{L_{B}}(u) \, du = 
\frac{t}{2}\text{csch} \left( \frac{t}{2} \right)
\end{equation}
\noindent
assuming the symmetry of $L_{B}$. The result now follows from entry 
$3.982.1$ in \cite{gradshteyn-2007a} 
\begin{equation}
\int_{-\infty}^{\infty} \text{sech}^{2}(au) \, \cos(tu ) \, du = 
\frac{\pi t}{a^{2}} \text{csch} \left( \frac{\pi t}{2a} \right).
\end{equation}
\end{proof}

\begin{note}
The integral representation of the Bernoulli polynomials
\begin{equation}
B_{n}(x) = \mathbb{E}\left[ ( x - \tfrac{1}{2} + i L_{B} )^{n} \right].
\label{umbra-ber}
\end{equation}
\noindent
is a special case of Theorem \ref{density-B}. The 
formula \eqref{umbra-ber} is stated in non-probabilistic language as 
\begin{equation}
B_{n}(x) = \frac{\pi}{2} \int_{-\infty}^{\infty} 
\left( x - \tfrac{1}{2} + it \right)^{n} \text{ sech}^{2}(\pi t)  \, dt.
\label{soliton-1}
\end{equation}
\noindent
To the best of our knowledge, this evaluation first 
appeared in \cite{touchard-1956a}. The 
role played by $\text{sech}^{2}x$ as a solitary wave for the Kortweg-de 
Vries equation has prompted the titles of the evaluations of \eqref{soliton-1}
in \cite{boyadzhiev-2007a,grosset-2005a}.
\end{note}

The next result uses Theorem \ref{density-B} to evaluate a 
nonlinear function of the Bernoulli polynomials that will be needed 
later. More examples will 
appear in the companion paper \cite{dixit-2012b}.

\begin{theorem}
\label{umbra-psi}
Let $\psi(x) = \Gamma'(x)/\Gamma(x)$ be the digamma function. Then
\begin{equation}
{\rm{eval}} \left\{ \log \mathfrak{B}(x) \right\} = 
\psi \left( \tfrac{1}{2} + \left| x - \tfrac{1}{2} \right| \right).
\end{equation}
\noindent
In particular, for $x \geq \tfrac{1}{2}$, 
\begin{equation}
{\rm{eval}} \left\{ \log \mathfrak{B}(x) \right\} = \psi(x).
\end{equation}
\end{theorem}
\begin{proof}
Theorem \ref{density-B} with $h(x) = \log x$ gives 
\begin{equation}
{\rm{eval}} \left\{ \log \mathfrak{B}(x) \right\} = 
\mathbb{E} \left[ \log \left( x - \tfrac{1}{2} + i L_{B} \right) \right]. 
\end{equation}
\noindent
The density $f_{L_{B}}$ is an even function, therefore the random 
variables $L_{B}$ and $-L_{B}$ 
have the same distribution. This gives 
\begin{eqnarray*}
{\rm{eval}} \left\{ \log \mathfrak{B}(x) \right\} & = & 
\frac{1}{2} \mathbb{E} \left[ 
\log \left( (x - \tfrac{1}{2})^{2}  + L_{B}^{2} \right) \right]
\\
& = &  \log \left( x - \tfrac{1}{2} \right) + 
\frac{1}{2} \mathbb{E} \left[ 
\log \left( 1 + \frac{L_{B}^{2}}{(x - \tfrac{1}{2})^{2}}
\right) \right]. 
\end{eqnarray*}
\noindent
A linear scaling of entry $4.373.4$ in \cite{gradshteyn-2007a} gives 
\begin{equation}
\int_{0}^{\infty} \frac{\log(1 + bu^{2})}{\sinh^{2} c u} \, du =
\frac{2}{c} \left[ \log \frac{c}{\pi \sqrt{b}} - 
\frac{\pi \sqrt{b}}{2c} - 
\psi \left( \frac{c}{\pi \sqrt{b}} \right) \right],
\vm
\end{equation}
\noindent 
for $b, \, c > 0$. Define
\begin{equation}
h(b,c) := 
\frac{2}{c} \left[ \log \frac{c}{\pi \sqrt{b}} - 
\frac{\pi \sqrt{b}}{2c} - 
\psi \left( \frac{c}{\pi \sqrt{b}} \right) \right]
\end{equation}
\noindent
and observe that 
\begin{eqnarray*}
\int_{0}^{\infty} \frac{\log(1 + bu^{2})}{\sinh^{2} 2 \pi u} \, du & = & 
\frac{1}{4} \int_{0}^{\infty} \frac{\log(1 + bu^{2})}{\cosh^{2} \pi u} 
\, \frac{du}{\sinh^{2} \pi u} \\
& = & 
\frac{1}{4} \int_{0}^{\infty} \frac{\log(1 + bu^{2})}{\cosh^{2} \pi u} 
\, \left( \frac{\cosh^{2} \pi u}{\sinh^{2}\pi u} - 1 \right) du.
\end{eqnarray*}
\noindent
It follows that 
\begin{equation}
\int_{0}^{\infty} \frac{\log(1 + bu^{2})}{\cosh^{2} \pi u} \, du = 
h(b, \pi) - 4 h(b, 2 \pi).
\end{equation}
\noindent
Now take $b = (x- \tfrac{1}{2})^{-2}$ to produce 
\begin{eqnarray}
\mathbb{E} \log \left( 1 + bL_{B}^{2} \right) 
& = & 
\frac{\pi}{2} \int_{0}^{\infty} 
\frac{\log \left( 1 + b u^{2}  \right)}{\cosh^{2} \pi u}  \, du 
\label{rep-1} \\
& = & \frac{\pi}{2} \left( h(b,\pi) - 4 h(b, 2 \pi) \right) \nonumber \\
& = & \log \frac{1}{\sqrt{b}} - 2 \log \frac{2}{\sqrt{b}} - 
\psi \left( \frac{1}{\sqrt{b}} \right) + 
2 \psi \left( \frac{2}{\sqrt{b}} \right). \nonumber
\end{eqnarray}

\noindent
The duplication formula
\begin{equation}
\psi(2z) = \tfrac{1}{2}\psi(z) + 
\tfrac{1}{2} \psi \left( z + \tfrac{1}{2} \right)
 + \log 2,
\vm
\end{equation}
\noindent
that appears as entry $8.365.6$ in \cite{gradshteyn-2007a}, reduces 
\eqref{rep-1} to the stated form.
\end{proof}

\section{The periodicity of the modified  Bernoulli numbers $B_{2n+1}^{*}$}
\label{modified-ber} 
\setcounter{equation}{0}

This section uses the umbral method to express the generating function of the 
modified Bernoulli numbers $B_{n}^{*}$ in terms of the digamma function
$\psi(x)$. The 
periodicity of $B_{2n+1}^{*}$ in Theorem \ref{zagier-per6} 
follows from this computation. Zagier \cite{zagier-1998a} establishes 
this result by using the expression 
\begin{equation}
2 \sum_{n=1}^{\infty} B_{n}^{*} x^{n} = 
\sum_{r=1}^{\infty} \frac{B_{r}}{r} \frac{x^{r}}{(1-x)^{2r}} - 
2 \log(1-x).
\end{equation}
\noindent
In the proofs given here the generating function employed admits an explicit 
expression.

\begin{theorem}
\label{gf-bstar}
The generating function of the sequence $\{ B_{n}^{*} \}$ is given by 
\begin{equation}
F_{B^{*}}(z) :=\sum_{n=1}^{\infty} B_{n}^{*}z^{n} = - \frac{1}{2} \log z - 
\frac{1}{2} \psi \left( z + 1/z - 1  \right).
\end{equation}
\end{theorem}
\begin{proof}
Start with
\begin{eqnarray*}
F_{B^{*}}(z) & = & 
\sum_{n=1}^{\infty} z^{n}\sum_{r=0}^{\infty} \binom{n+r}{2r}\frac{B_{r}}{n+r} 
\\
& = & \sum_{n=1}^{\infty} z^{n}\sum_{r=1}^{\infty} 
\binom{n+r}{2r}\frac{B_{r}}{n+r}+
\sum_{n=1}^{\infty} z^{n}\binom{n}{0}\frac{B_{0}}{n}.
\end{eqnarray*}
The second term is
$-\log(1-z)$. Interchanging the order of summation in the first term gives 
\begin{equation*}
\sum_{n=1}^{\infty} z^{n}\sum_{r=1}^{\infty}
\binom{n+r}{2r}\frac{B_{r}}{n+r} = 
\sum_{r=1}^{\infty}\frac{B_{r}}{(2r)!}
\sum_{n=r}^{\infty}z^{n}\frac{(n+r-1)!}{(n-r)!}
\end{equation*}
and the inner sum is identified as
\begin{equation*}
\sum_{n=r}^{\infty} z^{n}\frac{\left(n+r-1\right)!}{\left(n-r\right)!}
=\sum_{m=0}^{\infty}z^{m+r}\frac{(2r+m-1)!)}{m!}
=\left(2r-1\right)!\frac{z^{r}}{\left(1-z\right)^{2r}}.
\end{equation*}
Therefore 
\begin{equation}
\label{fb-1}
F_{B^{*}}(z) = - \log(1-z) + 
\sum_{r=1}^{\infty}\frac{B_{r}}{2r}\frac{z^{r}}{\left(1-z\right)^{2r}}.
\end{equation}
\noindent
The rules of umbral calculus now give an 
expression for $F_{B^{*}}(z)$. The identity
\begin{equation*}
\sum_{r=1}^{\infty}\frac{B_{r}}{2r}\frac{z^{r}}{\left(1-z\right)^{2r}}
=- {\rm{eval}} \left\{ \frac{1}{2}\log\left(1-\frac{ z \mathfrak{B} }
{\left(1-z\right)^{2}}\right) \right\},
\end{equation*}
gives 
\begin{equation}
F_{B^{*}}(z) = -{\rm{eval}} \left\{ \frac{1}{2}\log\left(\left(1-z\right)^{2}-
z\mathfrak{B} \right)
\label{fb-2}
\right\}.
\end{equation}
\noindent
Further reduction yields 
\begin{eqnarray*}
\log\left(\left(1-z\right)^{2}-\B z\right) & = & 
\log z+\log\left(\frac{\left(1-z\right)^{2}}{z}-\B \right) \\
& = & \log z+\log\left(\frac{\left(1-z\right)^{2}}{z}+\B+1\right)\\
 & = & \log z+\log\left[\B\left(1+\frac{\left(1-z\right)^{2}}{z}\right)\right],
\end{eqnarray*}
\noindent
using \eqref{umbra-rel1}. The result now follows from Theorem \ref{umbra-psi}.
\end{proof}

The generating function of $B_{2n+1}^{*}$ is now obtained from Theorem 
\ref{gf-bstar}.  The 
proof presented here is similar to the one given in \cite{gessel-2003a}.

\begin{theorem}
\label{gf-odd11}
The generating function of the sequence of odd-order modified 
Bernoulli numbers is given by 
\begin{equation}
G_{B^{*}}(z) = \sum_{n=0}^{\infty} B_{2n+1}^{*}z^{2n+1} = 
\frac{3z^{11} - z^{9}-z^{7} + z^{5} + z^{3} - 3z}{4(z^{12}-1)}.
\end{equation}
\end{theorem}
\begin{proof}
Start with
\begin{equation}
\frac{F_{B^{*}}(z)-F_{B^{*}}(-z)}{2} = \sum_{n=0}^{\infty} B_{2n+1}^{*}z^{2n+1}.
\end{equation}
\noindent
To evaluate $F_{B^{*}}(-z)$ use the relation \eqref{fb-2} and the operational 
rule from Lemma \ref{cool-sum} to obtain
\begin{eqnarray*}
2F_{B^{*}}\left(-z\right) & = & 
-{\rm{eval}} \left\{ \log\left(\left(1+z\right)^{2}+z \B\right) \right\}\\
& = & -\log z-{\rm{eval}} \left\{ 
\log\left(\frac{\left(1+z\right)^{2}}{z}+ \B\right) \right\} \\
 & = & -\log z-{\rm{eval}} \left\{ 
\log\left[\B\left(\frac{\left(1+z\right)^{2}}{z}\right)\right] \right\} \\
& = & -\log z-\psi\left(\frac{\left(1+z\right)^{2}}{z}\right).
\end{eqnarray*}
Therefore 
\begin{eqnarray*}
G_{B^{*}}\left(z\right) & = & 
\frac{F_{B^{*}}\left(z\right)-F_{B^{*}}\left(-z\right)}{2} \\
& = & -\frac{1}{4}\left(\psi\left(1+\frac{\left(1-z\right)^{2}}{z}\right)-
\psi\left(\frac{\left(1+z\right)^{2}}{z}\right)\right)\\
 & = & \frac{1}{4}\psi\left(z+ 1/z+2 \right) -
\frac{1}{4}\psi\left(z + 1/z -1 \right).
\end{eqnarray*}
\noindent
Now use the relation (entry $8.365$ in \cite{gradshteyn-2007a})
\begin{equation}
\psi(z + m) = \psi(z) + \sum_{k=0}^{m-1} \frac{1}{z+k}
\end{equation}
to obtain the result.
\end{proof}

Theorem \ref{zagier-per6} is now obtained as a consequence of Theorem 
\ref{gf-odd11}.

\begin{corollary}
\label{period-6}
The sequence of odd-order modified Bernoulli 
numbers $B_{2n+1}^{*}$ is periodic of period $6$.
\end{corollary}

\section{An inversion formula for the modified  Bernoulli numbers}
\label{inv-modber1} 
\setcounter{equation}{0}

This section discusses an expression for the classical Bernoulli numbers $B_{n}$
in terms of the modified ones $B_{n}^{*}$. The result appears already in 
\cite{zagier-1998a}, but the proof presented here extends directly to the 
polynomial case as stated in Theorem \ref{inv-poly}. The 
details are simplified by introducing a minor 
adjustment of $B_{n}^{*}$.

\begin{lemma}
Define $\overline{B}_{n} = B_{n}^{*} - 1/n$. Then 
\begin{equation}
\overline{B}_{n} = \sum_{k=1}^{n} \binom{n+k-1}{n-k} \frac{B_{k}}{2k}.
\vm
\end{equation}
\end{lemma}
\begin{proof}
The definition of $B_{n}^{*}$ in \eqref{zagier-mod} produces 
\begin{equation*}
\overline{B}_{n} = \sum_{k=1}^{n} \binom{n+k}{2k} \frac{B_{k}}{n+k}.
\vm
\end{equation*}
\noindent
Then use 
\begin{equation*}
\binom{n+k}{2k} \frac{1}{n+k} = 
\frac{(n+k-1)!}{(2k-1)! (n-k)!} \frac{1}{2k}
\vm
\end{equation*}
\noindent 
to deduce the claim.
\end{proof}

The inversion result is stated next.

\begin{theorem}
\label{inv-modber}
The sequence of Bernoulli numbers $B_{n}$ are given in terms of the modified 
Bernoulli numbers $B_{n}^{*}$ by
\begin{equation*}
B_{n} = 2n \sum_{k=1}^{n} (-1)^{n+k} 
\left[ \binom{2n-1}{n-k} - \binom{2n-1}{n-k-1} \right] B_{k}^{*} +
2 (-1)^{n} \binom{2n-1}{n}, 
\vm
\end{equation*}
\noindent
for $n \geq 1$. 
\end{theorem}
\begin{proof}
The inversion formulas
\begin{equation*}
a_{n} = \sum_{k=0}^{n} \binom{n+p+k}{n-k}b_{k}, \text{ and }
b_{n} = \sum_{k=0}^{n} (-1)^{k+n} \left[ \binom{2n+p}{n-k} - 
\binom{2n+p}{n-k-1} \right]a_{k},
\end{equation*}
\noindent
are given in \cite[(23), p. 67]{riordan-1968a}. Applying it to the sequence 
$\overline{B}_{n}$ gives 
\begin{equation}
\frac{B_{n}}{2n} = \sum_{k=1}^{n} (-1)^{n+k} 
\left[ \binom{2n-1}{n-k} - \binom{2n-1}{n-k-1} \right]\overline{B}_{k}.
\end{equation}
\noindent
The result now follows from
\begin{equation}
\sum_{k=1}^{n} (-1)^{n+k} 
\left[ \binom{2n-1}{n-k} - \binom{2n-1}{n-k-1} \right] \frac{1}{k} = 
\frac{(-1)^{n}}{n} \binom{2n-1}{n}.
\label{nice-sum}
\end{equation}
\noindent
To prove the identity \eqref{nice-sum} write the summand as 
\begin{equation}
\frac{1}{k} \left[ \binom{2n-1}{n-k} - \binom{2n-1}{n-k-1} \right]  = 
\frac{1}{n} \binom{2n}{n-k}
\end{equation}
\noindent 
and convert \eqref{nice-sum} to 
\begin{equation}
\sum_{k=1}^{n} (-1)^{k} \binom{2n}{n-k} = -\binom{2n-1}{n}.  
\vm
\end{equation}
\noindent
This follows directly from the basic sum
\begin{equation}
\sum_{k=0}^{2n} (-1)^{j} \binom{2n}{j} = 0.
\end{equation}
\end{proof}

\section{A generating function for Zagier polynomials}
\label{gf-zagierpoly} 
\setcounter{equation}{0}

This section gives the generating function of the Zagier polynomials 
\begin{equation}
B_{n}^{*}(x) = \sum_{r=0}^{n} \binom{n+r}{2r} \frac{B_{r}(x)}{n+r}.
\label{poly-ber-st}
\end{equation}
\noindent
The proof is similar to
Theorem \ref{gf-bstar}, so just an outline of the proof is presented.

\begin{theorem}
\label{gf-bstar-poly}
The generating function of the sequence $\{ B_{n}^{*}(x) \}$ is given by 
\begin{equation}
F_{B^{*}}(x;z) =\sum_{n=1}^{\infty} B_{n}^{*}(x)z^{n} = - \frac{1}{2} \log z - 
\frac{1}{2} \psi \left( z + 1/z - 1 - x \right).
\end{equation}
\end{theorem}
\begin{proof}
The starting point is the polynomial variation of \eqref{fb-2} in the form
\begin{eqnarray*}
F_{B^{*}}(x;z) & = &  - {\rm{eval}} 
\left\{ \frac{1}{2} \log \left( (1-z)^{2} - z \B(x) 
\right) \right\} \\
 & = &  - {\rm{eval}} \left\{ 
\frac{1}{2} \log \left( (1-z)^{2} - zx - z \B \right) \right\}
\\
 & = &  - {\rm{eval}} \left\{ 
\frac{1}{2} \log \left( 1 - 2z + z^{2} -zx - z \B 
\right)  \right\} \\
 & = &  - \frac{1}{2} \log z - 
{\rm{eval}} \left\{ \frac{1}{2} \log \left( 1/z - 2 + z -x -  \B 
\right) \right\}.
\end{eqnarray*}
\noindent
Now use $- \B = \B + 1$ to obtain
\begin{equation}
F_{B^{*}}(x;z)  =  - \frac{1}{2} \log z - \frac{1}{2}
{\rm{eval}} \left\{ \log \left( 1/z +  z  - 1 - x + \B  \right) 
\right\}.
\end{equation}
\noindent
The final claim now follows from Theorem \ref{umbra-psi}.
\end{proof}

\begin{corollary}
\label{gf-bstaralt-poly}
The generating function of the sequence $\{ (-1)^{n}B_{n}^{*}(x) \}$ is 
given by 
\begin{equation}
F_{B^{*}}(x;-z) :=\sum_{n=1}^{\infty} (-1)^{n} B_{n}^{*}(x)z^{n} = 
- \frac{1}{2} \log z - 
\frac{1}{2} \psi \left( z + 1/z + 2 +  x \right).
\end{equation}
\end{corollary}
\begin{proof}
Replacing  $z$ by $-z$ in the third line of the proof of Theorem 
\ref{gf-bstar-poly} gives 
\begin{eqnarray*}
F_{B^{*}}(x;-z) & = & - {\rm{eval}} 
\left\{ \frac{1}{2} \log \left( 1 + 2 z + z^{2} + 
zx + z \B \right) \right\} \\
 & = & - \frac{1}{2} \log z - {\rm{eval}} \left\{ \frac{1}{2} 
\log \left( z + 1/z + 2 + x + \B \right) \right\}.
\end{eqnarray*}
\noindent
As before, the result now comes from Theorem \ref{umbra-psi}.
\end{proof}

The next step is to provide analytic expressions for the generating 
functions of the subsequences $\{ B_{2n+1}^{*}(x) \}$ and 
$\{ B_{2n}^{*}(x) \}$. These formulas will be used in Section 
\ref{sec-periodic} to obtain information about these subsequences 
and in particular to discuss periodic subsequences in Theorem \ref{zag-thm1a}.

\begin{corollary}
\label{gf-bstar-poly-odd}
The generating functions of the odd and even parts of the sequence of 
Zagier polynomials are given by 
\begin{equation}
\sum_{n=0}^{\infty} B_{2n+1}^{*}(x)z^{2n+1} =
\frac{1}{4} \psi \left( z + 1/z +2+x  \right) - 
\frac{1}{4} \psi \left( z + 1/z -1 - x  \right),
\label{diff-psi}
\end{equation}
\noindent
and 
\begin{equation}
\sum_{n=1}^{\infty} B_{2n}^{*}(x)z^{2n} =
- \frac{1}{2} \log z 
-\frac{1}{4}  \psi \left( z + 1/z +2+x  \right) 
- \frac{1}{4} \psi \left( z + 1/z -1 - x  \right).
\label{sum-psi}
\end{equation}
\end{corollary}

The results of Corollary \ref{gf-bstar-poly-odd} correspond to the analogue 
of the ordinary generating function  for 
the Bernoulli polynomials. This is expressed in terms on the Hurwitz zeta
function as stated in Theorem \ref{thm-hurwitz}. The latter is defined by 
\begin{equation}
\zeta(s,a) = \sum_{n=0}^{\infty} \frac{1}{(n+a)^{s}}
\end{equation}
\noindent
and it has the integral representation (see \cite{temme-1996a}, page 76)
\begin{equation}
\zeta(s,a) = \frac{1}{\Gamma(s)} \int_{0}^{\infty} \frac{e^{-at} t^{s-1}}
{1 - e^{-t}} \, dt, \quad \realpart{s} > 1, \, \realpart{a} > 0.
\label{int-hurzeta}
\end{equation}
\noindent
The exponential generating function \eqref{gf-berpoly}, because it is given 
by an elementary function, is employed more 
frequently.

\begin{theorem}
\label{thm-hurwitz}
The generating function of the Bernoulli polynomials $B_{n}(x)$ is given by
\begin{equation}
\sum_{n=0}^{\infty} B_{n}(x) z^{n} = \frac{1}{z} \zeta \left( 2, 1/z - x + 1
\right).
\end{equation}
\end{theorem}
\begin{proof}
The integral representation of the gamma function 
\begin{equation}
\Gamma(s) = \int_{0}^{\infty} u^{s-1}e^{-u} \, du, \quad \realpart{s} > 0
\end{equation}
\noindent
and the special value $\Gamma(n+1) = n!$ give 
\begin{equation}
\sum_{n=0}^{\infty} B_{n}(x) z^{n} = 
\int_{0}^{\infty} e^{-u} \sum_{n=0}^{\infty} \frac{B_{n}(x)}{n!} (zu)^{n} \, du.
\end{equation}
\noindent
The generating function \eqref{gf-berpoly} is used to produce 
\begin{equation}
\sum_{n=0}^{\infty} B_{n}(x) z^{n} = 
z \int_{0}^{\infty} \frac{u}{1 - e^{-zu}} e^{-(1-xz+z)u}  \, du.
\end{equation}
\noindent
The change of variables $v = zu$ and \eqref{int-hurzeta} complete the proof.
\end{proof}

\section{Some arithmetic questions}
\label{sec-arithmetic} 
\setcounter{equation}{0}

There is marked difference in the arithmetical behavior of the numbers 
$B_{n}^{*}(j)$ according to the parity of $n$. For instance 
\begin{equation*}
\{ B_{n}^{*}(0): \, 1 \leq n \leq 10 \} = 
\left\{ \frac{3}{4}, \, \frac{1}{24}, \, - \frac{1}{4}, \, - \frac{27}{80}, \, 
- \frac{1}{4}, \, - \frac{29}{1260}, \, \frac{1}{4}, \, 
\frac{451}{1120} \, \frac{1}{4}, \, - \frac{65}{264} \right\}
\end{equation*}
\noindent
and
\begin{equation*}
\{ B_{n}^{*}(1): \, 1 \leq n \leq 10 \} = 
\left\{ \frac{5}{4}, \, \frac{25}{24}, \,  \frac{5}{4}, \,  \frac{133}{80}, \, 
\frac{9}{4}, \,  \frac{3751}{1260}, \, \frac{15}{4}, \, 
\frac{4931}{1120} \, \frac{19}{4}, \, \frac{1255}{264} \right\}.
\end{equation*}
\noindent
On the other hand, keeping $n$ fixed and varying $j$ gives 
\begin{equation*}
\{ B_{1}^{*}(j): \, 1 \leq j \leq 10 \} = 
\left\{ \frac{5}{4}, \, \frac{7}{4}, \,  \frac{9}{4}, \,  \frac{11}{4}, \, 
 \frac{13}{4}, \,  \frac{15}{4}, \, \frac{17}{4}, \, 
\frac{19}{4} \, \frac{21}{4}, \, \frac{23}{4} \right\}
\end{equation*}
\noindent
and 
\begin{equation*}
\{ B_{2}^{*}(j): \, 1 \leq j \leq 10 \} = 
\left\{ \frac{25}{24}, \, \frac{61}{24}, \,  \frac{109}{24}, \, 
 \frac{169}{24}, \, 
 \frac{241}{24}, \,  \frac{325}{24}, \, \frac{421}{24}, \, 
\frac{529}{24} \, \frac{649}{24}, \, \frac{781}{24} \right\}.
\end{equation*}
\noindent
This suggests that every element in the list
$\{ B_{n}^{*}(j): \, j \geq 1 \}$ has a denominator that is 
\textit{independent} of $j$, therefore this value is also 
the denominator of the 
modified Bernoulli number $B_{n}^{*}$. Assume that this is true and define 
$\alpha(n)$ be this common value; that is,
\begin{equation}
\alpha(n) = \text{denominator}(B_{n}^{*}).
\end{equation}
\noindent
As usual, the parity of $n$ plays a role in the results. 

\smallskip

The next theorem shows, for the case $n$ odd, that the function $\alpha(n)$
is well defined.

\begin{theorem}
For $j \in \mathbb{Z}$, the values $4B_{2n+1}^{*}(j)$ are integers. That is,
\begin{equation}
\alpha(2n+1) = 4.
\end{equation}
\end{theorem}
\begin{proof}
The generating function \eqref{gf-bstar-poly-odd} gives 
\begin{align*}
\sum_{n=0}^{\infty}4B_{2n+1}^{*}(j)z^{2n+1}&=
\psi\left(z+1/z+j+2\right)-
\psi\left(z+1/z-j-1\right) \nonumber\\
&=\psi\left(z+1/z\right)+
\sum_{k=0}^{j+1}\frac{1}{z+1/z+k}-
\psi\left(z+1/z\right)+ \nonumber \\
& \qquad \qquad + \sum_{k=0}^{j}\frac{1}{z+1/z-j-1+k} \nonumber\\
&=\sum_{k=0}^{j+1}\frac{1}{z+ 1/z+k}+
\sum_{k=0}^{j+1}\frac{1}{z+1/z-j-1+k}-\frac{1}{z+1/z}.
\end{align*}
Replace $k$ by $j+1-k$ in the second sum to obtain
\begin{align*}
\sum_{n=0}^{\infty}4B_{2n+1}^{*}(j)z^{2n+1}&=
\sum_{k=0}^{j+1}\left(\frac{1}{z+1/z+k}+\frac{1}{z+ 1/z-k}\right)-
\frac{1}{z+ 1/z} \nonumber\\
&=2z\sum_{k=0}^{j+1}\frac{z^2+1}{(z^2+1)^2-k^2z^2}-\frac{z}{z^2+1}\nonumber\\
&=2z\sum_{k=1}^{j+1}\frac{z^2+1}{(z^2+1)^2-k^2z^2} + 
\frac{z}{z^{2}+1}.
\end{align*}
This implies $4B_{2n+1}^{*}(j)\in\mathbb{Z}$.
\end{proof}

\begin{note}
Arithmetic questions for the numbers $B_{2n}^{*}(j)$ seem to be more delicate.
The values $\alpha(2n)$ seem to be divisible by $4$ and the 
list of  $\frac{1}{4} \alpha(2n)$ begins with 
\begin{equation*}
\{ 6, \, 20, \, 315, \, 280, \, 66, \, 3003, \, 78, \, 9520, \, 
305235, \, 20900, \, 138, \, 19734, \, 6, \, 7540, \, 15575175 \},
\end{equation*}
\noindent
for $1 \leq n \leq 15$. The exact 
power of a prime $p$ that divides $\alpha(2n)$ exhibits some 
interesting patterns. For instance, Figure
\ref{den-2} shows this function for $p=2$. 

{{
\begin{figure}[ht]
\begin{center}
\centerline{\epsfig{file=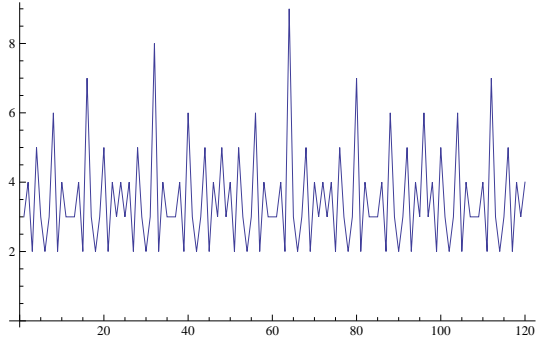,width=20em,angle=0}}
\caption{Power of $2$ that divides denominator of $B_{2n}^{*}(j)$}
\label{den-2}
\end{center}
\end{figure}
}}

\noindent
The data suggests that the  prime factors of $\alpha(2n)$ are bounded 
by $2n+1$. These questions will be addressed in a future paper.
\end{note}

\section{Some auxiliary binomial sums}
\label{sec-binsums} 
\setcounter{equation}{0}

This section contains the proofs of two identities for some sums involving 
binomial coefficients. These sums will be used in the Section 
\ref{sec-derivatives} to give an expression for the derivatives of Zagier 
polynomials. The identities given here are established using the method 
of creative telescoping described in \cite{petkovsek-1996a}. 

\begin{lemma}
\label{part-1}
For $n \in \mathbb{N}$,
\begin{equation}
\sum_{r=1}^{n-1} (-1)^{r} \frac{2(r+1)}{n+r+1} 
\binom{2r-1}{r} \binom{n+r+1}{2r+2} 
= -\left\lfloor \frac{n}{2} \right\rfloor.
\vm
\end{equation}
\end{lemma}
\begin{proof}
The summand on the left-hand side is written as 
\begin{equation*}
F(n,r) = \frac{(-1)^{r} \, (n+r)!}{2(2r+1) r!^{2} (n-r-1)!}.
\end{equation*}
\noindent
Observe that $F(n,r)$ vanishes for $r<0$ or $r>n-1$. The method of 
Wilf-Zeilberger lends the companion function 
\begin{equation}
G(n,r) = \frac{(-1)^{r+1} (n+r)!}{(n+1) (r-1)!^{2} (n-r+1)!}
\end{equation}
\noindent
together with the second order recurrence 
\begin{equation}
F(n+2,r) - F(n,r) = G(n,r+1)-G(n,r).
\end{equation}
\noindent
Sum both sides over all integers $r$ and check that the right-hand sum vanishes
to produce 
\begin{equation}
\sum_{r \in \mathbb{Z}} F(n+2,r) = \sum_{r \in \mathbb{Z}} F(n,r).
\label{allsums}
\end{equation}
\noindent
Define 
\begin{equation}
f(n) = \sum_{r=0}^{n-1} F(n,r).
\end{equation}
\noindent
Then \eqref{allsums} gives $f(n+2) = f(n)$. The initial conditions 
$f(1) = 1/2$ and $f(2) = 0$ show that 
\begin{equation}
f(n) = \begin{cases}
 \tfrac{1}{2} & \quad \text{ for } n \text{ odd } \\
 0 & \quad \text{ for } n \text{ even. }
\end{cases}
\end{equation}
\noindent
The desired sum starts at $r=1$, so its value is $f(n) - F(n,0)$. Thus, 
$F(n,0) = n/2$ gives the result.
\end{proof}

\begin{lemma}
\label{part-2}
For $n \in \mathbb{N}$ and $1 \leq k \leq n-1$, define 
\begin{equation}
u(n,k) = \sum_{r=k}^{n-1} \frac{2(-1)^{r}r(r+1)}{n+r+1} \binom{n+r+1}{2r+2} 
\left[ \binom{2r-1}{r-k} - \binom{2r-1}{r-k-1} \right].
\label{def-unk}
\end{equation}
\noindent
Then for $n$ even  
\begin{equation}
u(n,k) = \begin{cases} 
    - k & \quad \text{ for } k \text{ odd}, \\
    0 & \quad \text{ for } k \text{ even}
\end{cases}
\label{un-1}
\end{equation}
\noindent
and for $n$ odd
\begin{equation}
u(n,k) = \begin{cases} 
    0 & \quad \text{ for } k \text{ odd}, \\
    k & \quad \text{ for } k \text{ even}.
\end{cases}
\label{un-2}
\end{equation}
\end{lemma}
\begin{proof}
A routine binomial simplification gives 
\begin{equation}
u(n,k) = k \sum_{r=k}^{n-1} (-1)^{r}  \binom{n+r}{2r+1} \binom{2r}{r-k}.
\vm
\end{equation}
\noindent
This motivates the definition $\bar{u}(n,k) = u(n,k)/k$ and the assertion 
in \eqref{un-1} and \eqref{un-2} amounts to showing
\begin{equation}
\bar{u}(n,k) = \begin{cases} 
    +1 & \quad \text{ for } n \text{ odd}, k \text{ even} \\
    -1 & \quad \text{ for } n \text{ even}, k \text{ odd} \\
     0 & \quad \text{ otherwise.}
\end{cases}
\end{equation}
\noindent
The proof is similar to the one presented for 
Lemma \ref{part-1}. Introduce the functions 
$F(n,r,k) = (-1)^{r} \binom{n+r}{2r+1} \binom{2r}{r-k}$ and use the 
WZ-method to find the function 
\begin{equation}
G(n,r,k) = \frac{2 (n+1)(2r-1)(-1)^{r+1} }{(n+k+1)(n-k+1)} 
\binom{n+r}{2r-1} \binom{2r-2}{r-k-1}
\end{equation}
\noindent
companion to $F$ and the equation 
\begin{equation}
F(n+2,r,k) - F(n,r,k) = G(n,r+1,k) - G(n,r,k).
\end{equation}
\noindent
The argument is completed as before.
\end{proof}

\section{The derivatives of Zagier polynomials}
\label{sec-derivatives} 
\setcounter{equation}{0}

Differentiation of the 
generating function for Bernoulli polynomials 
\eqref{gf-berpoly} gives the relation 
\begin{equation}
\frac{d}{dx} B_{n}(x) = nB_{n-1}(x).
\label{der-berpoly0}
\end{equation}
\noindent
This section presents the analogous result for the Zagier polynomials.  The 
proof employs an expression for $B_{n}(x)$ in terms of $B_{n}^{*}(x)$; that is 
the inversion of \eqref{zag-polydef}. The proof is 
identical to that of Theorem \ref{inv-modber}, so
it is omitted.

\begin{theorem}
\label{inv-poly}
The sequence of Bernoulli polynomial $B_{n}(x)$ is given in terms of 
the Zagier polynomials  $B_{n}^{*}(x)$ by
\begin{equation*}
B_{n}(x) = 2n \sum_{k=1}^{n} (-1)^{n+k} 
\left[ \binom{2n-1}{n-k} - \binom{2n-1}{n-k-1} \right] B_{k}^{*}(x) +
2 (-1)^{n} \binom{2n-1}{n},
\vm
\end{equation*}
\noindent
for $n \geq 1$. 
\end{theorem}

The analogue of \eqref{der-berpoly0} is established next.

\begin{theorem}
\label{der-mod-zag}
The derivatives of the Zagier polynomials satisfy the relation
\begin{equation*}
\frac{d}{dx} B_{n}^{*}(x) = \sum_{j=1}^{\lf \frac{n}{2} \rf} 
(2j-1)
B_{2j-1}^{*}(x)  \quad \text{ for } n \text{ even}
\end{equation*}
\noindent
and 
\begin{equation*}
\frac{d}{dx} B_{n}^{*}(x) = \frac{1}{2} + \sum_{j=1}^{\lf \frac{n}{2} \rf} 
2j B_{2j}^{*}(x)  \quad \text{ for } n \text{ odd}.
\end{equation*}
\noindent
\end{theorem}
\begin{proof}
Differentiating \eqref{poly-ber-st} and using \eqref{der-berpoly0} gives 
\begin{eqnarray*}
\frac{d}{dx} B_{n}^{*}(x) 
& = &  \sum_{r=0}^{n-1} \binom{n+r+1}{2r+2} \frac{r+1}{n+r+1}
B_{r}(x) \vm \\
& = &  \frac{n}{2} + \sum_{r=1}^{n-1} \binom{n+r+1}{2r+2} \frac{r+1}{n+r+1}
B_{r}(x) \vm
\end{eqnarray*}
\noindent
for $n \geq 1$. The sum above (without the term $n/2$) is 
transformed using Theorem \ref{inv-poly} 
to produce 
\begin{multline}
2 \sum_{r=0}^{n-1} (-1)^{r} \binom{2r-1}{r} 
\binom{n+r+1}{2r+2} \frac{r+1}{n+r+1}  +  \vm \\
2 \sum_{r=1}^{n-1} 
\frac{r(r+1)}{n+r+1}  \binom{n+r+1}{2r+2} 
\sum_{k=1}^{r} (-1)^{k+r} 
\left[ \binom{2r-1}{r-k} - \binom{2r-1}{r-k-1} \right] B_{k}^{*}(x).
\end{multline}
\noindent
Denote the first sum by $S_{1}$ and the second one by $S_{2}$. 

\smallskip

Lemma \ref{part-1} shows that $S_{1} = - \lf \frac{n}{2} \rf$. To 
evaluate $S_{2}$, reverse  the order of summation to obtain
\begin{equation}
S_{2} = \sum_{k=1}^{n-1} (-1)^{k}u(n,k)B_{k}^{*}(x),
\end{equation}
\noindent
with $u(n,k)$ defined in \eqref{def-unk}. This gives 
\begin{equation}
\frac{d}{dx} B_{n}^{*}(x) = \frac{n}{2} - \lf \frac{n}{2} \rf + 
\sum_{k=1}^{n-1} (-1)^{k} u(n,k) B_{k}^{*}(x),
\vm
\end{equation}
\noindent
and the proof now follows from the values of $u(n,k)$ given in 
Lemma \ref{part-2}.
\end{proof}

\section{Some basics on Chebyshev polynomials}
\label{sec-chebyshev} 
\setcounter{equation}{0}

This section contains some basic information about the Chebyshev polynomials 
of first and second kind, denoted 
by $T_{n}(x)$ and $U_{n}(x)$, respectively. These properties will be used 
to establish some results on Zagier polynomials and the relation between these 
two families of polynomials will be clarified in 
Section \ref{sec-representation}.

The Chebsyhev polynomials of the first kind $T_{n}$ are defined by 
\begin{equation}
T_{n}(\cos \theta) = \cos n \theta, \quad n \in \mathbb{N} \cup \{ 0 \}
\label{form-tn1}
\end{equation}
\noindent
and the companion Chebyshev polynomial of the second kind $U_{n}$ by 
\begin{equation}
U_{n} ( \cos \theta)  = \frac{\sin( (n+1) \theta)}{\sin \theta}.
\label{form-un1}
\end{equation}
\noindent
Their generating functions are given by 
\begin{equation}
\sum_{n=0}^{\infty} T_{n}(x) t^{n} = \frac{1 - xt}{1- 2 xt + t^{2}}
\label{gf-tn}
\end{equation}
\noindent
and 
\begin{equation}
\sum_{n=0}^{\infty} U_{n}(x) t^{n} = \frac{1}{1- 2 xt + t^{2}}.
\label{gf-un}
\end{equation}
\noindent
The even part of this series, given by
\begin{equation}
\sum_{n=0}^{\infty} U_{2n}(x) t^{2n} = 
\frac{1+t^{2}}{1 + 2(1-2x^{2})t^{2} + t^{4}},
\label{gf-uneven}
\end{equation}
\noindent
will be used in Section \ref{mod-ber-poly-reflection}. Many 
properties of these families may be found in \cite{olver-2010a}. The 
formulas for the generating functions also appear in 
\cite[$22:3:8$, page 199]{spanier-1987a}. 

\smallskip

Differentiation of the generating function for $T_{n}(x)$ gives the 
basic identity 
\begin{equation}
\frac{d}{dx} T_{n}(x) = n U_{n-1}(x).
\label{der-tn}
\vm
\end{equation}
\noindent
The expression 
\begin{equation}
U_{n}(x) = \frac{( x + \sqrt{x^{2}-1} )^{n+1} - ( x - \sqrt{x^{2}-1})^{n+1}}
{2 \sqrt{x^{2}-1}},
\label{form-utwon}
\vm
\end{equation}
\noindent 
will be used in the arguments presented below. 

\smallskip

The Chebyshev polynomials have a hypergeometric representation in the form
\begin{eqnarray}
T_{n}(x) & = & {_{2}F_{1}} 
\left(-n,n; \tfrac{1}{2}; \tfrac{1-x}{2} \right) \label{hyper-T} \vm \\
U_{n}(x) & = & (n+1) \, {_{2}F_{1}} 
\left(-n,n+2; \tfrac{3}{2}; \tfrac{1-x}{2} \right). \vm
\nonumber
\end{eqnarray}
\noindent
These appear in \cite[p. 394, equation (15.9.5) and(15.9.6)]{olver-2010a}. 

\smallskip

The next statement is an expression of the Chebyshev polynomial $T_{n}$
that will be used to establish, in Theorem \ref{thm-symm}, a 
symmetry property of the Zagier polynomials.

\begin{lemma}
\label{cheby-12}
The Chebyshev polynomial $T_{n}(x)$ satisfies
\begin{equation}
\sum_{r=0}^{n} \binom{n+r}{2r} \frac{x^{r}}{n+r} = 
\frac{1}{n} T_{n} \left( \frac{x}{2} + 1 \right).
\vm
\label{cheby-11}
\end{equation}
\noindent
\end{lemma}
\begin{proof}
The representation \eqref{hyper-T} and 
$\left( \tfrac{1}{2} \right)_{r} = 2^{-2r}(2r)!/r!$ yield
\begin{eqnarray*}
T_{n} \left( \frac{x}{2} + 1 \right) & = & 
 {_{2}F_{1}} \left( -n, \, n; \, \tfrac{1}{2}, \, - \tfrac{x}{4} 
\right) \vm \\
& = &  \sum_{r=0}^{n} \frac{(-n)_{r} (n)_{r} }{( \tfrac{1}{2} 
)_{r}} \frac{(-x/4)^{r}}{r!} \vm \\
& = &  \sum_{r=0}^{n} \{ n(n-1) \cdots (n-(r-1)) \} \, 
\{ n(n+1) \cdots (n+r-1) \} \,  \frac{x^{r}}{(2r)!} \vm \\
& = & n \sum_{r=0}^{n} \frac{(n-r)! \, \{ (n-r+1)(n-r+2) \cdots (n+r-1) \}}
{(2r)! \, (n-r)!} x^{r} \\
& = & n \sum_{r=0}^{n} \frac{(n+r-1)!}{(2r)! \, (n-r)!} x^{r}. \vm
\end{eqnarray*}
\noindent
This verifies the claim. 
\end{proof}

\begin{lemma}
\label{cheby-umbral}
The Zagier polynomial $B_{n}^{*}(x)$ is related to the  Chebyshev polynomial 
$T_{n}(x)$ via 
\begin{equation*}
B_{n}^{*}(x)  =  {\rm{eval}} \left\{ \frac{1}{n} T_{n} \left( \frac{\B(x)}{2} 
+ 1 \right) \right \}  
  =  {\rm{eval}} \left\{ \frac{1}{n} T_{n} \left( \frac{\B + x + 2}{2} 
\right) \right \}. 
\end{equation*}
\end{lemma}
\begin{proof}
This is simply the umbral version of \eqref{cheby-11}.
\end{proof}

\begin{lemma}
\label{modi-ver1}
The Zagier polynomial $B_{n}^{*}(x)$ is given by 
\begin{equation}
B_{n}^{*}(x) = \frac{1}{n} \mathbb{E} \left[ T_{n} \left( \frac{x}{2} + 
\frac{1}{2} iL_{B} + \frac{3}{4} \right) \right].
\end{equation}
\end{lemma}
\begin{proof}
The result now follows from Lemma \ref{cheby-umbral}, 
Theorem \ref{density-B} and the umbral rule \eqref{umbra-rel1}.
\end{proof}

The result of Lemma \ref{modi-ver1} is now used to give an umbral proof of 
Theorem \ref{der-mod-zag}. 

\begin{proof}
The computation is simpler with $\tilde{B}_{n}(x) = B_{n}^{*} \left( x - 
\tfrac{3}{2} \right)$. Differentiate 
the statement of Lemma \ref{modi-ver1} to obtain
\begin{equation}
\frac{d}{dx}\tilde{B}_{n}(x)  =  
\frac{1}{2n} \mathbb{E} \left[ T_{n}' \left(  \frac{x+ iL_{B}}{2} \right) 
\right] 
 =  \frac{1}{2} \mathbb{E} \left[ U_{n-1} \left(  \frac{x+ iL_{B}}{2} \right) 
\right].
\end{equation}
\noindent
In the case of even degree, this gives 
\begin{equation}
\frac{d}{dx}\tilde{B}_{2n}(x)  =  
\frac{1}{2} \mathbb{E} \left[ U_{2n-1} \left(  \frac{x+ iL_{B}}{2} \right) 
\right],
\end{equation}
\noindent
and using the identity 
\begin{equation}
U_{2n-1}(x) = 2 \sum_{k=1}^{n} T_{2k-1}(x)
\end{equation}
\noindent
that is entry $18.18.33$ in \cite{olver-2010a}, it follows that 
\begin{equation}
\frac{d}{dx}\tilde{B}_{2n}(x)   =  
\mathbb{E} \left[ \sum_{k=1}^{n} T_{2k-1} 
\left(  \frac{x+ iL_{B}}{2} \right) 
\right]  = 
\sum_{k=1}^{n} (2k-1) \tilde{B}_{2k-1}(x),
\end{equation}
\noindent
as claimed.  The same argument works for $n$ odd using
\cite[18.18.32]{olver-2010a}:
\begin{equation}
U_{2n}(x) = 2 \sum_{k=0}^{n} T_{2k}(x) - 1.
\end{equation}
\end{proof}

\section{The Zagier-Chebyshev connection}
\label{sec-representation} 
\setcounter{equation}{0}

In \cite{zagier-1998a}, after the proof of the identity 
\begin{equation}
2 B_{2n}^{*} = \left( \frac{-3}{n} \right) + 
\sum_{r=0}^{n} (-1)^{n+r} \binom{n+r}{2r} \frac{B_{2r}}{n+r},
\label{zagier-88}
\end{equation}
\noindent
the author remarks that the second term has a \textit{pleasing similarity} to 
the equation \eqref{zagier-mod}. This 
section contains a representation of the Zagier polynomials 
$B_{n}^{*}(x)$ in terms of the Chebyshev polynomials of the 
second kind $U_{n}(x)$. The expressions contain terms that also have
pleasing similarity to the definition of Zagier polynomials. 
The results are naturally divided according to the parity of $n$.

\begin{theorem}
\label{thm-zagier-even}
The Zagier polynomials are given by
\begin{equation}
2B_{2n}^{*}(x) = \sum_{r=0}^{n} (-1)^{n+r} \binom{n+r}{2r} 
\frac{B_{2r}(x)}{n+r} + U_{2n-1} \left( \frac{x}{2} \right) +
U_{2n-1} \left( \frac{x+1}{2} \right)
\label{zag-pol1}
\vm
\end{equation}
\noindent
and 
\begin{equation}
2B_{2n+1}^{*}(x) = \sum_{r=0}^{n} (-1)^{n+r} \binom{n+r+1}{2r+1} 
\frac{B_{2r+1}(x)}{n+r+1} + U_{2n} \left( \frac{x}{2} \right) +
U_{2n} \left( \frac{x+1}{2} \right).
\label{zag-pol2}
\vm
\end{equation}
\end{theorem}
\begin{proof}
The proof is presented for the even degree case, the proof is similar for 
odd degree.

Theorem \ref{gf-bstar-poly} gives the generating function for $B_{n}^{*}(x)$.
Its even part yields 
\begin{equation*}
2 \sum_{n=1}^{\infty} B_{2n}^{*}(x)z^{2n} = 
- \frac{1}{2} \log z - \frac{1}{2} \psi( 1/z +z -x-1) 
- \frac{1}{2} \log (-z) - \frac{1}{2} \psi( -1/z -z -x-1).
\end{equation*}
\noindent
The functional equation $\psi(t+1) = \psi(t) + 1/t$ gives 
\begin{multline}
\label{nice-22}
2 \sum_{n=1}^{\infty} B_{2n}^{*}(x) z^{2n}  =   H(x,z)  + H(x,-z) \\
\qquad \qquad  +   \frac{1}{2} \left( 
\frac{1}{1/z + z + x + 3 } 
+ \frac{1}{1/z + z + x + 2 } 
- \frac{1}{1/z + z - x - 3 } 
- \frac{1}{1/z + z - x - 2 } \right)
\end{multline}
\noindent
with 
\begin{equation}
H(x,z) = - \frac{1}{2} \left(  \log z + \psi(1/z + z - x  - 3 ) \right).
\end{equation}
\noindent
The umbral method and Theorem \ref{umbra-psi} give 
\begin{eqnarray*}
2H(x,z) & = & - \log z - {\rm{eval}} \left( \log( 1/z +z -x -3 + \B ) \right) \\
 & = & - \log z - {\rm{eval}} \left( \log( 1/z +z -x -4 - \B ) \right) \\
 & = &  - {\rm{eval}} \left( \log( 1+ z^{2} - 4z - zx - x \B ) \right) \\
 & = &  - \log( 1+ z^{2}) - 
{\rm{eval}} \left( \log( 1 - \frac{z \B(x+4)}{1+z^{2}} ) \right) \\
 & = &  - \log( 1+ z^{2}) +
{\rm{eval}} \left( \sum_{r=1}^{\infty} \frac{ (z \B(x+4) )^{r}}{r(1+z^{2})^{r}}
 \right) \\
& = & - \log(1+z^{2}) + \sum_{r=1}^{\infty} 
\frac{B_{r}(x+4) z^{r}}{r(1+z^{2})^{r}}. 
\end{eqnarray*}
\noindent
Therefore
\begin{equation*}
H(x,z) + H(x,-z) = - \log(1+z^{2}) 
+ \sum_{r=1}^{\infty} \frac{B_{2r}(x+4) \, z^{2r}}{2r(1+z^{2})^{2r}}.
\end{equation*}

\noindent
Now observe that 
\begin{eqnarray*}
\sum_{r=1}^{\infty} \frac{B_{2r}(x+4) \, z^{2r}}{2r(1+z^{2})^{2r}} 
 & = & \sum_{r=1}^{\infty} \frac{z^{2r} B_{2r}(x+4)}{2r} 
\sum_{n=0}^{\infty} \frac{(2r)_{n} (-z^{2})^{n}}{n!}  \\
& = & \sum_{r=1}^{\infty} \frac{(-1)^{r}B_{2r}(x+4)}{2r} 
\sum_{n=r}^{\infty} (-1)^{n} \binom{n+r-1}{2r-1} z^{2n} \\
& = & \sum_{n=1}^{\infty} \left( \sum_{r=1}^{n} (-1)^{n+r} 
\binom{n+r}{2r} \frac{B_{2r}(x+4)}{n+r} \right)z^{2n}.
\end{eqnarray*} 
\noindent
This produces
\begin{equation*}
\sum_{r=1}^{\infty} \frac{B_{2r}(x+4) \, z^{2r}}{2r(1+z^{2})^{2r}} 
- \log(1+z^{2})  =  
\sum_{n=1}^{\infty} \left( \sum_{r=0}^{n} (-1)^{n+r} 
\binom{n+r}{2r} \frac{B_{2r}(x+4)}{n+r} \right)z^{2n}.
\end{equation*}

The equation \eqref{nice-22} now gives 
\begin{multline}
\label{nice-23}
2 \sum_{n=1}^{\infty} B_{2n}^{*}(x) z^{2n} = 
\sum_{n=1}^{\infty} \left( \sum_{r=0}^{n} (-1)^{n+r} \binom{n+r}{2r} 
\frac{B_{2r}(x+4)}{n+r} \right) z^{2n} \\
\qquad \qquad  +   \frac{1}{2} \left( 
\frac{1}{1/z + z + x + 3 } 
+ \frac{1}{1/z + z + x + 2 } 
- \frac{1}{1/z + z - x - 3 } 
- \frac{1}{1/z + z - x - 2 } \right).
\end{multline}

The generating function for $U_{n-1}(x)$, given in \eqref{gf-un}, is 
written as 
\begin{equation}
\sum_{n=1}^{\infty} U_{n-1}(x)z^{n} = \frac{1}{1/z + z - 2x}
\end{equation}
\noindent
and the rational function appearing in \eqref{nice-23} is expressed as 
\begin{equation*}
\sum_{n=1}^{\infty} \left( 
U_{n-1} \left( \frac{-x-3}{2} \right) 
+U_{n-1} \left( \frac{-x-2}{2} \right) 
-U_{n-1} \left( \frac{x+3}{2} \right) 
-U_{n-1} \left( \frac{x+2}{2} \right)  \right) z^{n}.
\end{equation*}
\noindent
Using the fact that $U_{n}(x)$ has the same parity as $n$ it simplifies to
\begin{equation*}
2 \sum_{n=1}^{\infty} \left( 
U_{2n-1} \left( \frac{-x-3}{2} \right) 
+U_{2n-1} \left( \frac{-x-2}{2} \right) \right) z^{2n}.
\end{equation*}
\noindent
Therefore
\begin{equation*}
2B_{2n}^{*}(x) = \sum_{r=0}^{n} (-1)^{n+r} \binom{n+r}{2r} 
\frac{B_{2r}(x+4)}{n+r} + 
U_{2n-1} \left( \frac{-x-3}{2} \right) 
+U_{2n-1} \left( \frac{-x-2}{2} \right).
\end{equation*}
\noindent
Finally, replace $x$ by $-x-3$,  use Theorem \ref{thm-symm} and the 
symmetry of Bernoulli polynomials $B_{2n}(1-x)  = B_{2n}(x)$, to obtain the 
result.
\end{proof}

The next result gives a representation for the difference of two 
Zagier polynomials in 
terms of the Chebyshev polynomials $U_{n}(x)$.

\begin{lemma}
The Zagier polynomials satisfy 
\begin{equation}
B_{n}^{*}(x+1) = B_{n}^{*}(x) + \frac{1}{2}U_{n-1} \left( \frac{x}{2} + 1 
\right).
\label{nice-reader}
\end{equation}
\noindent
It can be extended to 
\begin{equation}
B_{n}^{*}(x) - B_{n}^{*}(x-k) = \frac{1}{2} \sum_{j=1}^{k} 
U_{n-1} \left( \frac{x-j}{2} + 1 \right).
\label{nice-reader2}
\end{equation}
\end{lemma}
\begin{proof}
Apply Lemma \ref{taylor} and the representation of $B_{n}^{*}(x)$ in 
Lemma \ref{cheby-umbral}.
\end{proof}

\begin{note}
The sum in \eqref{zag-pol1} equals 
\begin{equation}
\sum_{r=0}^{n} (-1)^{n+r} \binom{n+r}{2r} \frac{B_{2r}(x)}{n+r} = 
2 B_{2n}^{*}(x-2).
\label{nice-two}
\vm
\end{equation}
\noindent
The proof of this fact is given below. First, use it to express 
\eqref{zag-pol1} as 
\begin{equation}
B_{2n}^{*}(x) - B_{2n}^{*}(x-2)  = \frac{1}{2} \left( 
U_{2n-1} \left( \frac{x}{2} \right) +
U_{2n-1} \left( \frac{x+1}{2} \right)
\right).
\vm
\end{equation}
\noindent
In this form, it can be extended directly to 
\begin{equation}
B_{2n}^{*}(x) - B_{2n}^{*}(x-2k) = \frac{1}{2} \left( 
\sum_{j=0}^{k-1} 
U_{2n-1} \left( \frac{x-2j}{2} \right) +
U_{2n-1} \left( \frac{x-2j+1}{2} \right)
\right).
\vm
\end{equation}

The proof of \eqref{nice-two} is given next. The identity 
\begin{equation}
T_{2n}(x) = (-1)^{n} T_{n}(1 - 2x^{2})
\end{equation}
\noindent
that appears in \cite[7.2.10(7), page 550]{brychkov-2008a} is used in the 
proof. Start with
\begin{eqnarray*}
\sum_{r=0}^{n} (-1)^{n+r} \binom{n+r}{2r} \frac{B_{2r}(x)}{n+r} & = & 
{\rm{eval}} \left\{ (-1)^{n} \sum_{r=0}^{n} \binom{n+r}{2r} 
\frac{( - \B^{2}(x) )^{r}}{n+r} \right\} \\
& = & {\rm{eval}} \left\{ \frac{(-1)^{n}}{n} 
T_{n} \left( - \frac{\B^{2}(x)}{2} + 1 \right) \right\}  \\
& = & {\rm{eval}} \left\{ \frac{(-1)^{n}}{n} 
T_{n} \left( 1 - 2  \left(\frac{\B(x)}{2}  \right)^{2} \right) \right\}  \\
& = & 2 \, {\rm{eval}} \left\{ \frac{1}{2n} 
T_{2n} \left( \frac{\B(x)}{2} \right) \right\}  \\
& = & 2 \, {\rm{eval}} \left\{ \frac{1}{2n} 
T_{2n} \left( \frac{\B(x)-2}{2} +1 \right) \right\}  \\
& = & 2 \, {\rm{eval}} \left\{ \frac{1}{2n} 
T_{2n} \left( \frac{\B(x-2)}{2} +1 \right) \right\}  \\
& = & 2 \sum_{r=0}^{2n} \binom{2n+r}{2r} \frac{B_{r}(x-2)}{2n+r} \\
& = & 2 B_{2n}^{*}(x-2).
\end{eqnarray*}
\end{note}

\medskip

A special case of Theorem \ref{thm-zagier-even} gives a simple proof of 
\eqref{zagier-88}.

\begin{corollary}
\label{coro-zag1}
The modified Bernoulli numbers $B_{2n}^{*}$ are given by 
\begin{equation}
2B_{2n}^{*} = \left( \frac{-3}{n} \right) + 
\sum_{r=0}^{n} (-1)^{n+r} \binom{n+r}{2r} \frac{B_{2r}}{n+r}.
\vm
\end{equation}
\noindent
Here $\left( \frac{ \bullet}{n} \right)$ is the Jacobi symbol. 
\end{corollary}
\begin{proof}
Let $x=0$ in Theorem \ref{thm-zagier-even}, use the value $U_{2n-1}(0)=0$ 
and observe that 
\begin{equation}
U_{2n-1} \left( \tfrac{1}{2} \right) = 
\begin{cases}
1 & \quad \text{ if } n \equiv 1 \bmod 3,  \\
-1 & \quad \text{ if } n \equiv -1 \bmod 3,  \\
0 & \quad \text{ if } n \equiv 0 \bmod 3, 
\end{cases}
\vm
\end{equation}
\noindent
can be written as 
$U_{2n-1} \left( \tfrac{1}{2} \right) = \left( \frac{-3}{n} \right)$.
\end{proof}

The next result appears in \cite{zagier-1998a}.

\begin{corollary}
\label{coro-zag2}
Let $n \in \mathbb{N}$. Then
\begin{equation}
B_{2n}^{*} + n = \sum_{r=0}^{2n} \binom{2n+r}{2r} \frac{(-1)^{r} B_{r}}{2n+r}.
\vm
\label{nice-99}
\end{equation}
\end{corollary}
\begin{proof}
The right-hand side of \eqref{nice-99} is $B_{2n}^{*}(1)$. Therefore, the 
statement becomes $B_{2n}^{*} + n = B_{2n}^{*}(1)$. This is established by
letting $x=0$ in \eqref{nice-reader} and the value 
%to obtain 
%\begin{eqnarray*}
%2B_{2n}^{*}(1) &  = & \sum_{r=0}^{n} (-1)^{n+r} \binom{n+r}{2r} 
%\frac{B_{2r}(1)}{n+r} + U_{2n-1} \left( \frac{1}{2} \right) + 
%U_{2n-1}(1) \vm \\
%& = & \sum_{r=0}^{n} (-1)^{n+r} \binom{n+r}{2r} 
%\frac{B_{2r}}{n+r} + \left( \frac{-3}{n} \right) + 
%U_{2n-1}(1) \vm \\
%& = & 2B_{2n}^{*} + U_{2n-1}(1), \vm
%\end{eqnarray*}
%\noindent
%using Corollary \ref{coro-zag1} in the last step. The final step uses the 
%limiting case of \eqref{form-un1} 
\begin{equation}
U_{2n-1}(1) = \lim\limits_{\theta \to 0} \frac{\sin 2 n \theta}{\sin \theta} 
= 2n.
\vm
\end{equation}
\end{proof}

Theorem \ref{thm-zagier-even} is now used to produce another proof of 
Theorem \ref{zagier-per6}. 

\begin{corollary}
The modified Bernoulli numbers $B_{2n+1}^{*}$ are given by 
\begin{equation}
B_{2n+1}^{*} = \frac{(-1)^{n}}{4} + \frac{1}{2} U_{2n}( \tfrac{1}{2}) = 
\frac{(-1)^{n}}{4} + \frac{1}{\sqrt{3}} \sin \left( \frac{(2n+1)\pi}{3} 
\right).
\label{value-per1}
\vm
\end{equation}
\noindent
In particular, $B_{2n+1}^{*}$ is periodic of period $6$.
\end{corollary}
\begin{proof}
Put $x=0$ in \eqref{zag-pol2} and observe that only the term $r=0$ 
survives in the sum. Now use the value $U_{2n}(0) = (-1)^{n}$ and let 
$\theta = \pi/3$ in 
\eqref{form-un1} to get the result.
\end{proof}

\begin{corollary}
\label{coro-twoder}
For $n \in \mathbb{N}$
\begin{equation}
2 B_{2n+1}^{*} \left( \tfrac{1}{2} \right) = U_{2n} \left( \tfrac{1}{4} \right)
+  U_{2n} \left( \tfrac{3}{4} \right).
\vm
\end{equation}
\end{corollary}
\begin{proof}
Let $x = \tfrac{1}{2}$ in \eqref{zag-pol2} and use the fact that 
$B_{j}(\tfrac{1}{2} ) = - (1- 2^{1-j}) B_{j}$.
\end{proof}

\section{A reflection symmetry of the Zagier polynomials}
\label{mod-ber-poly-reflection} 
\setcounter{equation}{0}

The classical Bernoulli polynomials $B_{n}(x)$ exhibit symmetry with 
respect to the line $x = \tfrac{1}{2}$ in the form 
\begin{equation}
B_{n}(1-x) = (-1)^{n}B_{n}(x).
\vm
\end{equation}
\noindent
This section describes the corresponding property for the Zagier 
polynomials: their symmetry is with 
respect to the line $x = -\tfrac{3}{2}$. 

\begin{theorem}
\label{thm-symm}
The Zagier polynomials satisfy the relation 
\begin{equation}
B_{n}^{*}(-x-3) = (-1)^{n} B_{n}^{*}(x).
\end{equation}
\end{theorem}
\begin{proof}
The first proof uses the generating function $F_{B^{*}}(x,z)$. Replacing 
$(x,z)$ by $(-x-3,-z)$ in the second line of the proof of Lemma
\ref{gf-bstar-poly} gives 
\begin{eqnarray*}
F_{B^{*}}(-x-3,-z) 
 & = &  - {\rm{eval}} \left\{ 
\frac{1}{2} \log \left( (1+z)^{2} - z(x+3) + z \B \right) \right\}
\\
 & = &  - \frac{1}{2} \log z - {\rm{eval}} \left\{ 
\frac{1}{2} \log \left(z+ 1/z + \B -1 - x \right) \right\} \\
& = & F_{B^{*}}(x,z).
\end{eqnarray*}
\noindent
This proves the statement.
\end{proof}

\smallskip

A second proof of Theorem \ref{thm-symm} uses the expression for the 
Zagier polynomials $B_{n}^{*}(x)$ in terms of the Chebyshev polynomials 
$T_{n}(x)$. Indeed, using $T_{n}(z) = (-1)^{n}T_{n}(z)$, 
\begin{eqnarray*}
B_{n}^{*}(-x-3) & = & 
\frac{1}{n} T_{n} \left( 
\frac{-x-3+ \B}{2} + 1 \right) \\
& = & \frac{1}{n} T_{n} \left( 
\frac{-x+ \B}{2}  - \frac{1}{2} \right) \\
& = & \frac{(-1)^{n}}{n} T_{n} \left( 
\frac{x- \B}{2}  + \frac{1}{2} \right) \\
& = & \frac{(-1)^{n}}{n} T_{n} \left( 
\frac{x + \B + 1}{2}  + \frac{1}{2} \right) \\
& = & \frac{(-1)^{n}}{n} T_{n} \left( 
\frac{x + \B }{2}  + 1 \right) \\
& = & (-1)^{n} B_{n}^{*}(x).
\end{eqnarray*}

The rest of the section gives a third proof of Theorem \ref{thm-symm}.

\begin{proof}
The induction hypothesis states that
$B_{m}^{*}(-x-3) = (-1)^{m}B_{m}^{*}(x)$ for all $m < n$.  The discussion is 
divided according to the parity of $n$.

\smallskip

\noindent
\textit{Case 1}. For $n$ is even, Theorem \ref{der-mod-zag} gives 
\begin{eqnarray*}
\frac{d}{dx}B_{n}^{*}(-x-3) & = & - \sum_{j=1}^{n/2} (2j-1)B_{2j-1}^{*}(-x-3) 
\vm \\
& = & - \sum_{j=1}^{n/2} (2j-1)(-1)^{2j-1}B_{2j-1}^{*}(x) \vm \\
& = &  \sum_{j=1}^{n/2} (2j-1)B_{2j-1}^{*}(x) \vm \\
& = & \frac{d}{dx} B_{n}^{*}(x). \vm
\end{eqnarray*}
\noindent 
It follows that $B_{n}^{*}(-x-3)$ and $B_{n}^{*}(x)$ differ by a constant. 
Now evaluate at $x = - \tfrac{3}{2}$ to see that this constant vanishes.
  
\smallskip

\noindent
\textit{Case 2}. Now assume $n$ is odd. The previous argument now gives 
\begin{equation}
B_{n}^{*}(-x-3) = - B_{n}^{*}(x) + C_{n}
\label{const1}
\end{equation}
\noindent
for some constant $C_{n}$. It remains to show $C_{n}=0$.

Iterating the relation 
\begin{equation}
B_{n}(x+1) = B_{n}(x) + nx^{n-1}
\vm
\end{equation}
\noindent
gives 
\begin{equation}
B_{n}(x+3) = B_{n}(x) + nx^{n-1} + n(x+1)^{n-1} + n(x+2)^{n-1}.
\label{recu-nice11}
\vm
\end{equation}
\noindent
Replace $x = - \tfrac{3}{2}$ in \eqref{const1} and in \eqref{recu-nice11} 
and observe that 
\begin{equation}
C_{n} = 2 \sum_{r=0} \frac{\binom{n+r}{2r}}{n+r} 
\left[ B_{r} \left( \tfrac{1}{2} \right) - 
r \left( - \tfrac{3}{2} \right)^{r-1} - 
r \left( - \tfrac{1}{2} \right)^{r-1} \right].
\end{equation}
\noindent
Thus, to show $C_{n}=0$, it is required to prove 
\begin{equation}
\sum_{r=0}^{n} \frac{\binom{n+r}{2r}}{n+r} B_{r} \left( \tfrac{1}{2} \right)
= \sum_{r=0}^{n} \frac{\binom{n+r}{2r}}{n+r} 
\left[ r \left( - \tfrac{3}{2} \right)^{r-1}  + 
r \left( - \tfrac{1}{2} \right)^{r-1} \right].
\end{equation}
\noindent
The left-hand side is nothing but $B_{n}^{*} \left( \tfrac{1}{2} \right)$. 
The  right-hand side is $V_{n}'(-\tfrac{3}{2}) + V_{n}'(- \tfrac{1}{2})$, where 
\begin{equation}
V_{n}(x) = \sum_{r=0}^{n} \frac{\binom{n+r}{2r}}{n+r} x^{r}.
\end{equation}
\noindent
Lemma \ref{cheby-12} shows that  
\begin{equation}
V_{n}(x) = \frac{1}{n}T_{n} \left( \frac{x}{2} + 1 \right).
\end{equation}
\noindent
Hence it suffices to show that 
\begin{equation}
2B_{n}^{*} \left( \tfrac{1}{2} \right) = 
U_{n-1} \left( \tfrac{1}{4} \right) + U_{n-1} \left( \tfrac{3}{4} \right).
\label{special-case}
\end{equation}
\noindent
This is the result of Corollary \ref{coro-twoder}.
\end{proof}

\begin{note}
We note that  unlike \eqref{zag-pol1}, the formula in \eqref{zag-pol2}, of 
which \eqref{special-case} is a special case, does not use the 
symmetry $B_{2n+1}^{*}(-x-3) = - B_{2n+1}^{*}(x)$ in its proof.
\end{note}

\section{Values of Zagier polynomials that yield periodic sequences}
\label{sec-periodic} 
\setcounter{equation}{0}

The original observation of Zagier, that $B_{2n+1}^{*} = B_{2n+1}^{*}(0)$ 
is a periodic sequence 
(with period $6$ and values $\{ \tfrac{3}{4}, \, - \tfrac{1}{4}, 
\, - \tfrac{1}{4}, \, \tfrac{1}{4}, \, \tfrac{1}{4}, \, - \tfrac{3}{4} \}$) is 
extended here to other values of $x$. The first part of the discussion is to 
show that periodicity of $B_{2n+1}^{*}(x)$ implies that $2x$ is an 
integer.

\medskip

The discussion begins with an elementary statement.

\begin{lemma}
The sequence 
$\{ a_{n} \}$ is periodic, with minimal period $p$, if and only
if its generating function 
\begin{equation}
A(z) = \sum_{n=0}^{\infty} a_{n}z^{n} 
\end{equation}
\noindent
is a rational function of $z$ such that, when written in reduced form, the 
denominator has the form $D(z) = 1 - z^{p}$. 
\end{lemma}

\noindent
\textbf{Special values of $B_{2n+1}^{*}(x)$}. The case considered here 
discusses values of $x$ such that $\{ B_{2n+1}^{*}(x) \}$ is a periodic 
sequence. The generating function of this sequence is given in \eqref{diff-psi}
as 
\begin{equation}
\sum_{n=0}^{\infty} B_{2n+1}^{*}(x)z^{2n+1} =
\frac{1}{4} \psi \left( z + 1/z +2+x  \right) - 
\frac{1}{4} \psi \left( z + 1/z -1 - x  \right).
\label{diff-psi0}
\end{equation}

\begin{proposition}
\label{shift-b}
Let $b \in \mathbb{R}$ be fixed. Then 
\begin{equation}
\psi(t+b) - \psi(t) = R(t)
\label{psi-00}
\end{equation}
\noindent
for some rational function $R(t)$ if and only if $b \in \mathbb{Z}$.
\end{proposition}
\begin{proof}
Assume $b \in \mathbb{Z}$. It is clear that $b$ may be assumed positive. 
Iteration of $\psi(t+1) = \psi(t) + 1/t$ yields 
\begin{equation}
\psi(t+b) = \psi(t) + \sum_{k=0}^{b-1} \frac{1}{t+k}.
\label{usual-reduction}
\end{equation}
\noindent 
Therefore $\psi(t+b) - \psi(t)$ is a rational function. To prove 
the converse, assume \eqref{psi-00}
holds for some rational function $R$. Integrating both sides with respect to 
$t$ gives 
\begin{equation}
\ln \Gamma(t+b) - \ln \Gamma(t) = R_{1}(t) + \ln R_{2}(t) + C_{1}
\end{equation}
\noindent
for a pair of rational functions $R_{1}, \, R_{2}$ (coming from the 
integration of $R(t)$) and a 
constant of integration $C_{1}$. It follows that 
\begin{equation}
\frac{\Gamma(t+b)}{C_{2} R_{2}(t)\Gamma(t)} = e^{R_{1}(t)}.
\label{poles-0}
\end{equation}
\noindent
The singularities of the left-hand side are (at most) poles. On the 
other hand, the presence of a pole of $R_{1}(t)$ produces 
an essential singularity for the right-hand
side of \eqref{poles-0}. It follows that $R_{1}(t)$ is a polynomial. 
Comparing
the behavior of \eqref{poles-0}  as $t \to \pm \infty$ shows that $R_{1}$ 
must be a constant; that is,
\begin{equation}
\Gamma(t+b) = C_{3}R_{2}(t) \Gamma(t).
\label{last-pole}
\end{equation}
\noindent
The set equality
\begin{equation}
\{ b - k: \, k \in \mathbb{N} \} = 
\{ -k  : \, k \in \mathbb{N} \} \cup \{ t_{1}, \, t_{2}, \cdots, t_{r} \}
\end{equation}
\noindent
where $t_{i}$ are the poles of $R$ comes from comparing poles in 
\eqref{last-pole}.  Now take $k \in \mathbb{N}$ sufficiently
large so that $b - k \not = t_{i}$. Then $b - k = -k_{1}$ for some 
$k_{1} \in \mathbb{N}$. Therefore $b = k - k_{1} \in \mathbb{Z}$, as claimed. 
\end{proof}

The next lemma deals with the transition from the variable $z$ to 
$z + 1/z$.

\begin{lemma}
\label{rat-inv}
Assume $R(z)$ is a rational function that satisfies $R(z) = R(1/z)$. Then 
$R$ is a function of $1/z+z$ only.
\end{lemma}
\begin{proof}
It is assumed that 
\begin{equation}
R(z) = \frac{a_{0} + a_{1}z + \cdots + a_{n}z^{n}}
{b_{0} + b_{1}z + \cdots + b_{m}z^{m}} = 
\frac{a_{0} + a_{1}/z + \cdots + a_{n}/z^{n}}
{b_{0} + b_{1}/z + \cdots + b_{m}/z^{m}}.
\end{equation}
\noindent
Now use the fact that
\begin{equation}
\frac{u}{v} = \frac{U}{V} \text{ implies } 
\frac{u}{v} = \frac{U}{V} = 
\frac{u+U}{v+V}
\end{equation}
\noindent
to conclude that 
\begin{equation*}
R(z) = \frac{2a_{0} + a_{1}(z+ 1/z) + a_{2} (z^{2} +1/z^{2}) + \cdots +
a_{n}(z^{n} + 1/z^{n}) }
{2b_{0} + b_{1}(z+ 1/z) + b_{2} (z^{2} +1/z^{2}) + \cdots +
b_{m}(z^{m} + 1/z^{m}) }.
\end{equation*}
\noindent
The conclusion follows from the fact that $z^{j} + 1/z^{j}$ is a 
polynomial in $z+1/z$. This is given in entry $1.331.3$ 
of \cite{gradshteyn-2007a}.
\end{proof}

\begin{theorem}
\label{two-z}
The generating function 
\begin{equation}
\sum_{n=0}^{\infty} B_{2n+1}^{*}(x) z^{2n+1}
\label{nine-00}
\end{equation}
\noindent
is a rational function of $z$ if and only if $2x \in \mathbb{Z}$.
\end{theorem}
\begin{proof}
Assume \eqref{nine-00} is a rational function 
of $z$. Then \eqref{diff-psi} implies that 
\begin{equation}
\psi ( z + 1/z + 2 + x ) - \psi( z + 1/z -1 - x) = A(z)
\label{left-inv00}
\end{equation}
\noindent
with $A$ a rational function of $z$. The left-hand side of \eqref{left-inv00}
is invariant under $z \mapsto 1/z$, therefore Lemma \ref{rat-inv} shows that 
$A(z) = B( z + 1/z)$, for some rational function $B$. Now rewrite 
\eqref{left-inv00} as 
\begin{equation}
\psi(t + 2x + 3) - \psi(t) = B(t+1+x)
\end{equation}
\noindent
with $t = z + 1/z -1 -x$. Proposition \ref{shift-b} shows that $2x \in 
\mathbb{Z}$. 

To establish the converse, assume $2x \in \mathbb{Z}$. The identity 
\eqref{diff-psi} shows that 
\begin{equation}
4 \sum_{n=0}^{\infty} B_{2n+1}^{*}(x) z^{2n+1} = 
\psi( t + 2x+3) - \psi(t)
\end{equation}
\noindent
with $t = z + 1/z -1 - x$. Proposition \ref{shift-b} shows that 
$\psi(t + 2x+3) - \psi(t)$ is a rational function of $t$ and hence a 
rational function of $z$. 
\end{proof}

\begin{corollary}
Assume the sequence $\{ B_{2n+1}^{*}(x) \}$ is 
periodic. Then $2x \in \mathbb{Z}$.
\end{corollary}
\begin{proof}
The hypothesis implies that the generating function in \eqref{nine-00} is a
rational function. Theorem \ref{two-z} gives  the conclusion.
\end{proof}

\medskip

The quest for values of $x$ that produce periodic sequences $B_{2n+1}^{*}(x)$
is now reduced to the set $\mathbb{Z} \cup \left( \mathbb{Z} + 
\tfrac{1}{2} \right)$.  The symmetry given in Theorem \ref{thm-symm} implies 
that one may assume $x \leq - \tfrac{3}{2}$. 

\medskip

\subsection{Integer values of $x$} The nature of the sequence 
$\{ B_{2n+1}^{*}(x) \}$ is discussed next for $x = k \in \mathbb{Z}$. 

\begin{theorem}
\label{boddk}
Let $n \in \mathbb{N}$ and $k \geq 3$. Then
\begin{equation}
\label{value-k}
B_{2n+1}^{*}(-k) = - \frac{1}{4}U_{2n}(0) - \frac{1}{2} 
\sum_{j=1}^{k-2} U_{2n} \left( \frac{j}{2} \right).
\vm
\end{equation}
\end{theorem}
\begin{proof}
This is just a special case of \eqref{nice-reader2} with $x=0$. Use 
\eqref{value-per1} and the fact that $U_{2n}(0) = (-1)^{n}$. 
\end{proof}

The next step is to show that $\{B_{2n+1}^{*}(-k) \}$ is not periodic for 
$k \geq 5$. 

\begin{lemma}
Assume $j \geq 3$. Then $U_{2n} \left( \frac{j}{2} \right) > 0$.
\end{lemma}
\begin{proof}
This comes directly from \eqref{form-utwon}.
\end{proof}

%\begin{lemma}
%\label{asym}
%Let $n \in \mathbb{N}$. Then 
%\begin{equation}
%U_{2n} \left( \frac{3}{2} \right) = 
%\frac{1}{\sqrt{5}} \left[ \left( \frac{3 + \sqrt{5}}{2} \right)^{2n+1} - 
% \left( \frac{3 - \sqrt{5}}{2} \right)^{2n+1} \right].
%\vm
%\end{equation}
%\end{lemma}
%\begin{proof}
%The generating function \eqref{gf-un} gives 
%\begin{equation}
%\sum_{n=0}^{\infty} U_{2n} \left( \frac{3}{2} \right) t^{n} = 
%\frac{t+1}{t^{2}-7t+1}.
%\end{equation}
%\noindent
%The result now follows by expanding this rational function in partial 
%fractions.
%\end{proof}

\begin{proposition}
The sequence $\left\{ B_{2n+1}^{*}(-k) \right\}$ is not periodic for 
$k \geq 5$.
\end{proposition}
\begin{proof}
The identity \eqref{value-k} is written as 
\begin{eqnarray*}
-2B_{2n+1}^{*}(-k) & = & \frac{U_{2n}(0)}{2} + U_{2n}(1/2) + U_{2n}(1) + 
U_{2n}(3/2) + \sum_{j=4}^{k-2} U_{2n} \left( \frac{j}{2} \right) \vm \\
& \geq & 
\frac{U_{2n}(0)}{2} + U_{2n}(1/2) + U_{2n}(1) + 
U_{2n}(3/2). \vm
\end{eqnarray*}
\noindent
The value 
\begin{equation}
U_{2n} \left( \frac{3}{2} \right) = 
\frac{1}{\sqrt{5}} \left[ \left( \frac{3 + \sqrt{5}}{2} \right)^{2n+1} - 
 \left( \frac{3 - \sqrt{5}}{2} \right)^{2n+1} \right]
\label{special-un}
\vm
\end{equation}
shows that $\left\{ B_{2n+1}^{*}(-k) \right\}$ is not bounded. To obtain 
\eqref{special-un}, use $x = \tfrac{3}{2}$ in \eqref{form-utwon}.
\end{proof}

The next result shows that, after a linear modification, the case $k=-4$
produces another periodic example. 

\begin{proposition}
\label{prop-value4}
The sequence $\left\{ B_{2n+1}^{*}(-4) + n \right\}$ is $6$-periodic.
\end{proposition}
\begin{proof}
The value $k=4$ in \eqref{value-k} gives 
\begin{equation}
B_{2n+1}^{*}(-4) = - \frac{1}{4} U_{2n}(0) - 
\frac{1}{2}U_{2n} \left( \tfrac{1}{2} \right) - 
\frac{1}{2} U_{2n}(1).
\vm
\end{equation}
\noindent
The values $U_{2n}(0) = (-1)^{n}$ is $2$-periodic and 
\begin{equation}
U_{2n}(\tfrac{1}{2}) = \frac{2}{\sqrt{3}} \sin \left((2n+1)\pi/3 \right)
\end{equation}
is $3$-periodic (with values $0, \, -1, \, +1$). The expression 
\eqref{form-un1}, in the limit as $\theta \to 0$, gives
$U_{2n}(1) = 2n+1$. The proof is complete.
\end{proof}

\begin{corollary}
The sequence $\left\{ B_{2n+1}^{*}(-3) \right\}$ is $6$-periodic.
\end{corollary}
\begin{proof}
Choose $k = 3$ in \eqref{value-k} to obtain 
\begin{equation}
B_{2n+1}^{*}(-3) = - \frac{1}{4}U_{2n}(0) - \frac{1}{2}U_{2n} \left( 
\frac{1}{2} \right).
\vm
\end{equation}
\noindent
As in Proposition \ref{prop-value4}, $U_{2n}(0)$ is of period $2$ and 
$U_{2n} \left( \tfrac{1}{2} \right)$ is of period $3$. 
\end{proof}

\begin{proposition}
\label{value-2}
The sequence $\{ B_{2n+1}^{*}(-2) \}$
is $2$-periodic:
\begin{equation}
B_{2n+1}^{*}(-2) = \frac{(-1)^{n+1}}{4}.
\vm
\end{equation}
\end{proposition}
\begin{proof}
Let $k=2$ in Theorem \ref{boddk}.
%This is obtained by e identity in Lemma \ref{cheby-umbral} gives 
%\begin{eqnarray*}
%B_{2n+1}^{*}(-2) & = & 
% {\rm{eval}} \left\{ \frac{1}{2n+1} T_{2n+1} \left( \frac{\B}{2} \right)
%  \right\} \\
% & = & \varphi_{n}( \B -1 ) \\
% & = & \varphi_{n}(\B) - \varphi_{n}'(-1)
%\end{eqnarray*}
%\noindent
%where Proposition \ref{taylor} has been used in the last line. This yields 
%\begin{equation*}
%B_{2n+1}^{*}(-2) = \frac{1}{4}U_{2n}(0) - \frac{1}{2} U_{2n}(0) = 
%-\frac{1}{4}U_{2n}(0).
%\end{equation*}
%\noindent
%The result now follows from \eqref{form-un1}.
\end{proof}

The rest of the integer values $x$ are obtained by the symmetry rule 
given in Theorem \ref{thm-symm}. The study of the 
structure of the sequences $B_{2n+1}^{*}(k)$ has been completed. The 
details are summarized in the next statement.

\begin{theorem}
\label{thm-structure1}
Let $k \in \mathbb{Z}$. Then 

\noindent
a) $\{ B_{2n+1}^{*}(k) \}$ is exponentially unbounded 
if $k \geq 2$ or $k \leq -5$;

\noindent
b) $ \{B_{2n+1}^{*}(k) + n \}$ is $6$-periodic for $k=-4$ or $k=1$;

\noindent
c) $ \{ B_{2n+1}^{*}(k) \}$ is $6$-periodic if $k=-3$ or $k=0$;

\noindent
d) $ \{ B_{2n+1}^{*}(k) \}$ is $2$-periodic if $k=-2$ or $k=-1$.
\end{theorem}

\medskip

\subsection{Values of $x \in \tfrac{1}{2} + \mathbb{Z}$} The 
example $x = - \tfrac{3}{2}$ is considered first.

\begin{proposition}
For $n \in \mathbb{N} \cup \{ 0 \}$,
\begin{equation}
B_{2n+1}^{*} \left( - \tfrac{3}{2} \right) = 0.
\end{equation}
\end{proposition}
\begin{proof}
Theorem \ref{thm-symm} states that $B_{2n+1}^{*}(-x-3) = - B_{2n+1}^{*}(x)$. 
Replacing $x = - \tfrac{3}{2}$ gives the result.
\end{proof}

The symmetry of $B_{2n+1}^{*}(x)$ about $x = - \tfrac{3}{2}$ shows that it 
suffices to consider values of the form $k + \tfrac{1}{2}$ for $k \geq -1$. 

\begin{theorem}
\label{thm-khalf}
For all $k \geq -1$,
\begin{equation}
B_{2n+1}^{*} \left( k + \tfrac{1}{2} \right) = 
\frac{1}{2} \sum_{r=0}^{k+1} U_{2n} \left( \frac{2r+1}{4} \right).
\label{sum-khalf}
\vm
\end{equation}
\end{theorem}
\begin{proof}
The proof is similar to that of Theorem \ref{boddk}, so the details are 
omitted.
\end{proof}

The next lemma produces an unbounded value in the sum \eqref{sum-khalf} when 
$ k \geq 1$.

\begin{lemma}
\label{fivequarter}
For $n \in \mathbb{N}$
\begin{equation}
U_{2n} \left( \tfrac{5}{4} \right) = 
\frac{2^{2n+2} - 2^{-2n}}{3}.
\vm
\end{equation}
\end{lemma}
\begin{proof}
This comes directly from \eqref{form-utwon}.
\end{proof}

The next examples deal with values of $B_{2n+1}^{*}(k + \tfrac{1}{2})$ that 
do not contain the unbounded term $U_{2n} \left( \tfrac{5}{4} \right)$. 

\begin{lemma}
The sequence $B_{2n+1}^{*} \left( - \tfrac{1}{2} \right)$ is not periodic.
\end{lemma}
\begin{proof}
Theorem \ref{thm-khalf}, with $k = -1$,  and \eqref{form-un1} give
\begin{equation*}
B_{2n+1}^{*} \left( - \tfrac{1}{2} \right)  =   \tfrac{1}{2}
U_{2n} \left( \tfrac{1}{4} \right) 
 =  \frac{2}{\sqrt{15}} \sin \left( (2n+1) \theta \right),
\end{equation*}
\noindent
with $\cos \theta = \tfrac{1}{4}$. It follows from here that 
$\{ B_{2n+1}^{*}(-\tfrac{1}{2}) \}$ is not periodic. Indeed, if $p$ were a 
period, then 
$B_{2n+2p+1}^{*} \left( - \tfrac{1}{2} \right)  = 
B_{2n+1}^{*} \left( - \tfrac{1}{2} \right)$ implies 
\begin{equation}
\tan ( (2n+1) \theta ) = \cot p \theta \text{ for all } n \in \mathbb{N}.
\end{equation}
\noindent
Thus $3 \theta $ and $\theta$ must differ by an integer multiple 
of $\pi$; that is $2 \theta = \pi m$. This is impossible if $\cos \theta = 
\tfrac{1}{4}$. 
\end{proof}

\begin{lemma}
The sequence $B_{2n+1}^{*} \left(  \tfrac{1}{2} \right)$ is not periodic.
\end{lemma}
\begin{proof}
In the case $k=0$, Theorem \ref{thm-khalf} gives 
\begin{equation}
B_{2n+1}^{*} \left( \tfrac{1}{2} \right)  =   \tfrac{1}{2} \left[ 
U_{2n} \left( \tfrac{1}{4} \right) + U_{2n} \left( \tfrac{3}{4}  \right) 
\right]. 
\vm
\end{equation}
\noindent
To check that this is not a periodic sequence, use \eqref{gf-uneven} to
produce 
\begin{equation}
\sum_{n=0}^{\infty} \left[ U_{2n} \left( \tfrac{1}{4} \right) + 
U_{2n} \left( \tfrac{3}{4} \right) \right] t^{n} = 
\frac{8(1+t)(4t^2+3t+4)}{16t^4 + 24t^3+25t^2+ 24t + 16}.
\label{Y}
\end{equation}
\noindent
Periodicity of $B_{2n+1}^{*} \left( \tfrac{1}{2} \right)$ implies that 
the poles of of the right-hand side in \eqref{Y} must be roots 
of a polynomial of the form $1 - t^{p}$. In particular, the arguments of these 
poles must be  rational multiples of $\pi$. One of these poles is 
$t_{0} = (1+ 3 \sqrt{7}i)/8$, with argument
$\alpha = \cos^{-1} \left( \tfrac{1}{8} \right)$.  Therefore
$\alpha$ must be a rational multiple of $\pi$. To obtain a contradiction, 
observe that 
\begin{equation}
\omega_{m,n}:= 2 \cos \left( \frac{\pi m }{n} \right)
\end{equation}
\noindent
is a root of the monic polynomial  $2T_{n}(x/2)$. It follows that $\omega_{m,n}$
is an algebraic integer and a rational number (namely $\tfrac{1}{4}$). This 
implies that it must be an integer (see \cite[page 50]{stewarti-1979a}). This
is a contradiction.
\end{proof}

These results are summarized in the next theorem.

\begin{theorem}
\label{thm-structure2}
There is no integer value of $k \neq -2$ for which 
$\{ B_{2n+1}^{*}(k + \tfrac{1}{2}) \}$ is periodic.
\end{theorem}

\medskip

\noindent
\textbf{Special values of $B_{2n}^{*}(x)$}. The second case considered here 
deals with values of the subsequence $B_{2n}^{*}(x)$. Symbolic experiments
were unable to produce nice closed-forms for special values of 
$B_{2n}^{*}(x)$, but the identity 
\begin{equation}
B_{2n}^{*}(-1) = B_{2n}^{*}(-2)
\label{test-11}
\end{equation}
\noindent
motivated the definition of the function 
\begin{equation}
A_{2n}^{*}(u):= 
B_{2n}^{*}(-1-u) - B_{2n}^{*}(-1), \text{ for } u \in \mathbb{Z}.
\end{equation}

\begin{lemma}
For $n \in \mathbb{N}$, the function $A_{2n}^{*}(u)$ 
satisfies $A_{2n}^{*}(u) = A_{2n}^{*}(1-u)$.
Therefore, it suffices to describe $A_{2n}^{*}(u)$ for $u \geq 1$. 
\end{lemma}
\begin{proof}
This follows directly from Theorem \ref{thm-symm}.
\end{proof}

The next statement expresses the function $A_{2n}^{*}$ in terms of the 
Chebsyshev polynomials of the second kind $U_{2n-1}(x)$.

\begin{proposition}
The function $A_{2n}^{*}$ is given by 
\begin{equation}
A_{2n}^{*}(u) = \frac{1}{2} \sum_{j=2}^{u+1} U_{2n-1} \left( \frac{u+1-j}{2}
\right).
\end{equation}
\end{proposition}
\begin{proof}
Iterate the identity  \eqref{nice-reader}.
\end{proof}

The expression in \eqref{sum-psi} yields the next result.

\begin{lemma}
\label{bstar-even}
The generating function of the sequence $\{ B_{2n}^{*}(x) - 
B_{2n}^{*}(-1) \}$ satisfies
\begin{equation*}
4 \sum_{n=1}^{\infty} \left[ B_{2n}^{*}(x) - B_{2n}^{*}(-1) \right] z^{2n}  
  =   
- \psi(w -1-x)  - \psi(w+2+x) + 2 \psi(w) + \frac{1}{w}
\end{equation*}
\noindent
with $w = z + 1/z$.
\end{lemma}

The proof of the next result is similar to that of Theorem \ref{two-z}.

\begin{corollary}
\label{two-z1}
The generating function 
\begin{equation}
\sum_{n=1}^{\infty} \left[ B_{2n}^{*}(x) - B_{2n}^{*}(-1) \right] z^{n}
\end{equation}
\noindent
is a rational function of $z$ if and only if $2x \in \mathbb{Z}$.
\end{corollary}

The next statement is an analogue of Theorems \ref{thm-structure1} and 
\ref{thm-structure2}. 

\begin{theorem}
Let $A_{2n}^{*}(x) = B_{2n}^{*}(-1-x) - B_{2n}^{*}(-1)$ as above. Then 

\noindent
1) The sequences $A_{2n}^{*}(1)$ and $A_{2n}^{*}(0)$ vanish identically. 

\noindent
2) The sequences $A_{2n}^{*}(2)$ and $A_{2n}^{*}(-1)$ are periodic with period 3. The 
repeating values are $\{ \tfrac{1}{2}, \, - \tfrac{1}{2}, \, 0 \}$. 

\noindent
3) The sequences $A_{2n}^{*}(3)$ and $A_{2n}^{*}(-2)$ grow linearly in $n$. Moreover, 
$A_{2n}^{*}(3)-n$ and $A_{2n}^{*}(-2)-n$ are 
periodic with period 3. The 
repeating values are $\{ \tfrac{1}{2}, \, - \tfrac{1}{2}, \, 0 \}$. 

\noindent
4) The sequence $A_{2n}^{*}(x)$ is unbounded for $x \geq 4$ and $x \leq -3$.
\end{theorem}

\section{Additional properties of the Zagier polynomials}
\setcounter{equation}{0}
\label{sec-additional}

The Zagier polynomials $B_{n}^{*}(x)$ have a variety of interesting properties. 
These are recorded here for future studies. 

\smallskip 

\noindent
\textbf{Coefficients}. The Zagier polynomial $B_{n}^{*}(x)$ 
has rational coefficients, some of which are integers. Figure \ref{zag-1} 
shows the number of \textit{integer} coefficients in $B_{n}^{*}(x)$ 
as a function of $n$. The minimum
values seems to occur at the powers $2^{j}$, where the number of 
integer coeffcients is $j-1$.

{{
\begin{figure}[ht]
\begin{minipage}[t]{16em}
\centering
\epsfig{file=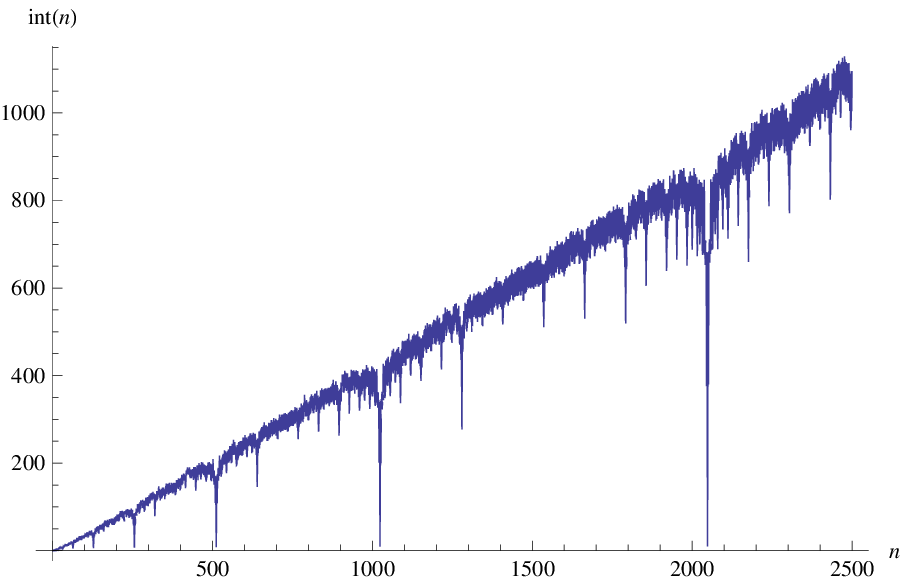,width=15em,angle=0}
\caption{Integer coefficients}
\label{zag-1}
\end{minipage}%
\begin{minipage}[t]{16em}
\centering
\epsfig{file=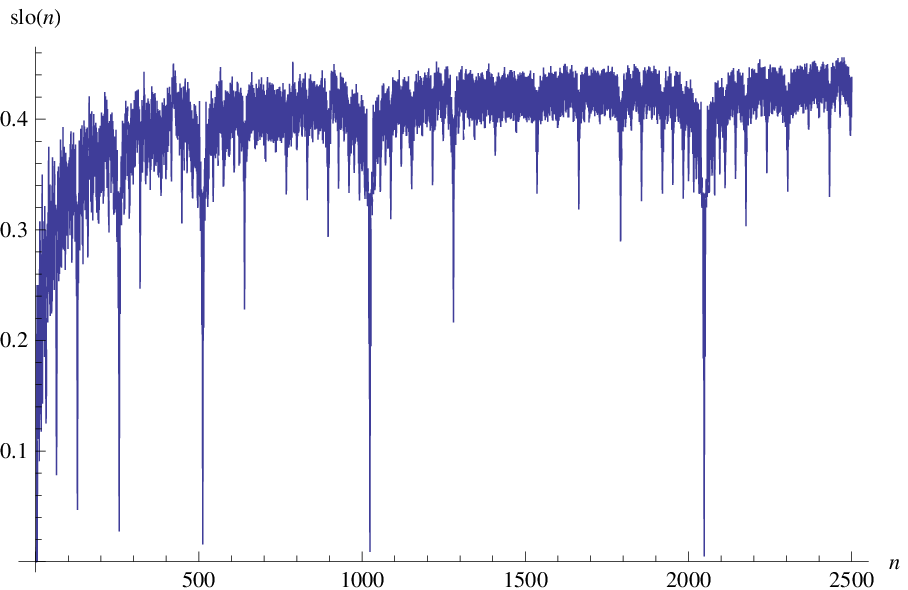,width=15em,angle=0}
\caption{Linear behavior}
\label{zag-2}
\end{minipage}
\end{figure}
}}

\medskip 

\noindent
\textbf{Signs of coefficients and shifts}. The coefficients of $B_{n}^{*}(x)$
do not have a fixed sign, but there is a tendency towards positivity. Figure 
\ref{excess} shows the excess of positive coefficients divided by the total 
number. On the other hand, the shifted polynomial $B_{n}^{*}(x + \tfrac{3}{2})$ 
appears to have 
only positive coefficients. The coefficients of the shifted polynomial 
appears to be \textit{logconcave}. This notion is defined in terms of the 
operator $\mathcal{L}$ acting on sequences 
$\{ a_{j} \}$ via $\mathcal{L}(\{ a_{j} \}) = \{ a_{j}^{2} - a_{j-1}a_{j+1} \}$.
A sequence is called logconcave if $\mathcal{L}(\{ a_{j} \})$ is nonnegative. 
The sequence is called \textit{infinitely logconcave} if any application of 
$\mathcal{L}$ produces positive sequences. The data suggests that the 
coefficients of $B_{n}^{*}(x + \tfrac{3}{2})$ form an infinitely logconcave 
sequence. 

\medskip

{{
\begin{figure}[ht]
\begin{minipage}[t]{16em}
\centering
\epsfig{file=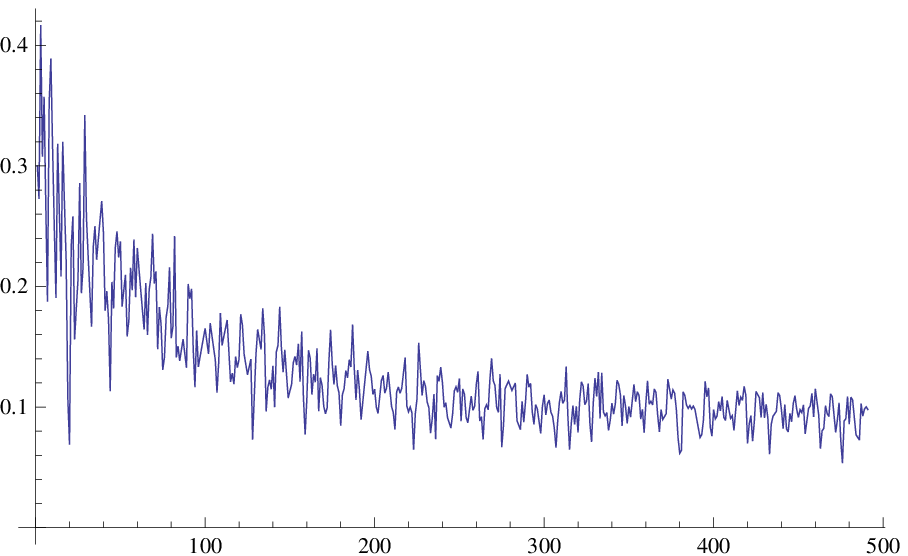,width=15em,angle=0}
\caption{Excess of positive coefficients}
\label{excess}
\end{minipage}%
\begin{minipage}[t]{16em}
\centering
\epsfig{file=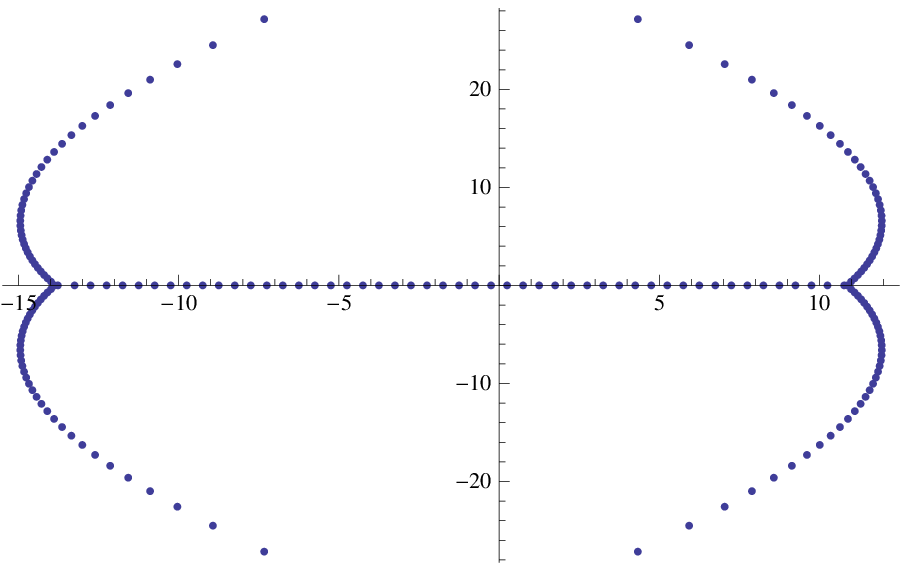,width=15em,angle=0}
\caption{Roots of $B_{200}^{*}(x)$.}
\label{roots-200}
\end{minipage}
\end{figure}
}}

\medskip

\noindent
\textbf{Roots of $B_{n}^{*}$}. There is a well-established connection 
between the nature of the roots of a polynomials and the logconcavity of 
its coefficients. P. Br\"and\'en \cite{branden-2011a} has shown that if a 
polynomial has only real and negative roots, then its sequence of coefficients
is infinitely logconcave. This motivated our computations of the 
roots of $B_{n}^{*}(x)$. The conclusion is that 
the polynomial 
$B_{n}^{*} \left( x + \tfrac{3}{2} \right)$ does not fall in this category
and Br\"and\'en's criteria does not apply. 
Figure \ref{roots-200} shows these roots for $n=200$. 

\smallskip

\noindent
\textbf{A second shift}. The polynomial $B_{n}^{*}(x - \tfrac{3}{2})$ admits 
a representation in terms of classical special functions. The 
Gegenbauer polynomial is defined by (see \cite{temme-1996a}, p. 152, (6.37)):
\begin{equation}
C_{n}^{(\lambda)}(x) = \binom{n+2 \lambda-1}{n} 
{_{2}F_{1}} \left( -n, \, n+2 \lambda; \, \lambda + \tfrac{1}{2}; \, 
\tfrac{1}{2}(1-x)  \right).
\end{equation}

\begin{theorem}
\label{thm-shifted-1}
The shifted Bernoulli polynomial $\tilde{B}_{n}(x)$, defined by 
$B_{n}^{*} \left( x - \tfrac{3}{2} \right)$ is given by
\begin{equation}
\tilde{B}_{n}(x) = \frac{1}{n} T_{n} \left( \frac{x}{2} \right) + 
\sum_{k=1}^{\lf \frac{n}{2} \rf} \frac{B_{2k}(1/2)}{k2^{2k+2}} 
C_{n-2k}^{(2k)} \left( \frac{x}{2} \right).
\end{equation}
\end{theorem}
\begin{proof}
Lemma \ref{modi-ver1} and expanding as a Taylor sum gives 
\begin{eqnarray*}
\tilde{B}_{n}(x) & = &  \mathbb{E} \left[ \frac{1}{n} T_{n} 
\left( \frac{x}{2} + \frac{1}{2} iL_{B} \right) \right] \\
& = & \frac{1}{n} \left( T_{n} \left( \frac{x}{2} \right) + 
\sum_{k=1}^{n} \frac{1}{k!} \mathbb{E} \left[ \left( \frac{i L_{B}}{2} \right)
^{k} \right]  \left( \frac{d}{dx} \right)^{k} T_{n} \left( \frac{x}{2} 
\right) \right).
\end{eqnarray*}
\noindent
The hypergeometric representation of the Chebyshev polynomial
\begin{equation}
T_{n}(x) = {_{2}F_{1}} \left( n, \, -n; \, \tfrac{1}{2}; \, \tfrac{1-x}{2} 
\right)
\end{equation}
\noindent
and the differentiation rule (Exercise 5.1 in \cite{temme-1996a}, p. 128)
\begin{equation}
\left( \frac{d}{dx} \right)^{k} {_{2}F_{1}} \left( a, \, b; \, c; \, z \right)
 = \frac{(a)_{k}(b)_{k}}{(c)_{k}} 
{_{2}F_{1}} \left( a+k, \, b+k; \, c+k; \, z \right)
\end{equation}
\noindent
give the identity 
\begin{equation}
\left( \frac{d}{dx} \right)^{k} T_{n}(x)  = n2^{k-1}(k-1)! C_{n-k}^{(k)}(x).
\end{equation}
\noindent
The odd moments of $L_{B}$ vanish and the even moments are given by 
\begin{equation}
\mathbb{E} \left[ \left( \frac{i L_{B}}{2} \right)^{2k} \right] = 
\frac{B_{2k}(\tfrac{1}{2})}{2^{2k}}
\end{equation}
\noindent
according to \eqref{umbra-ber}.
\end{proof}

The Chebyshev polynomial $T_{n}$ and the Gegenbauer polynomial 
$C_{n-2k}^{(2k)}$ have the same parity as $n$. Thus Theorem \ref{thm-shifted-1}
yields a new 
proof of Theorem \ref{thm-symm}, stated below in terms of $\tilde{B}_{n}$. 

\begin{corollary}
The shifted polynomials $\tilde{B}_{n}(x)$ have the same parity as $n$:
\begin{equation}
\tilde{B}_{n}(-x) = (-1)^{n} \tilde{B}_{n}(x).
\end{equation}
\end{corollary}

\section{The Euler case}
\setcounter{equation}{0}
\label{sec-euler}

This section describes a 
parallel treatment of the Euler polynomial $E_{n}(x)$ 
defined by the generating function 
\begin{equation}
\sum_{n=0}^{\infty} E_{n}(x) \frac{t^{n}}{n!} = \frac{2e^{tx}}{e^{t}+1},
\label{gen-fun-eulernum}
\end{equation}
\noindent
Their umbrae is 
\begin{equation}
{\rm{eval}} \left\{ \exp( t \E(x) ) \right \} = \frac{2e^{tx}}{e^{t}+1}.
\label{umbra-eulerpoly}
\end{equation}
\noindent
The Euler numbers are defined by 
\begin{equation}
\sum_{n=0}^{\infty} E_{n} \frac{t^{n}}{n!} = \frac{2e^{t}}{e^{2t}+1},
\label{gen-fun-euler}
\end{equation}
\noindent
and they appear as 
\begin{equation}
E_{n} = 2^{n} E_{n} \left( \tfrac{1}{2} \right).
\label{euler-form2}
\end{equation}
\noindent
Their umbra $\E$ is 
\begin{equation}
{\rm{eval}}( \exp(z \E) ) = {\rm{ sech }}(z)
\end{equation}
\noindent
and the Euler numbers are expressed as 
\begin{equation}
E_{n} = 2^{n} \left( \E + \tfrac{1}{2} \right)^{n},
\end{equation}
\noindent
which is an umbral equivalent of \eqref{euler-form2}.

The next statement is the analogue of Theorem \ref{density-B}. 

\begin{theorem}
\label{density-E}
There exists a real valued 
random variable $L_{E}$ with probability density $f_{L_{E}}(x)$
such that, for all admissible functions $h$, 
\begin{equation}
{\rm{eval}}\{ h(\mathfrak{E}(x)) \} = 
\mathbb{E} \left[ h( x - 1/2 + i L_{E}) \right]
\label{eval-98E}
\end{equation}
\noindent
where the expectation is defined in \eqref{int-expec}. The density 
of $L_{E}$ is given by 
\begin{equation}
f_{L_{E}}(x) = {\rm{sech}}(\pi x), \quad \text{ for } 
x \in \mathbb{R}.
\label{form-le}
\end{equation}
In particular,
\begin{equation}
{\rm{eval}}\{ \exp( \mathfrak{E}(x)) \} = 
\mathbb{E} \left[ it( x - 1/2 + i L_{E}) \right]
\end{equation}
\noindent
and 
\begin{equation}
E_{n}(x) = \mathbb{E}\left[ ( x - \tfrac{1}{2} + i L_{E} )^{n} \right].
\label{umbra-euler}
\end{equation}
\end{theorem}
%\begin{theorem}
%\label{density-E}
%Let $L_{E}$ be a random variable with density 
%\begin{equation}
%f_{L_{E}}(x) =  {\rm{sech}}(\pi x), \quad \text{ for } 
%x \in \mathbb{R}.
%\label{form-le}
%\end{equation}
%\noindent
%Then, for all admissible functions $h$, 
%\begin{equation}
%{\rm{eval}}\{ h(\mathfrak{E}(x)) \}
% = \mathbb{E} \left[ h( x - 1/2 + i L_{E}) \right].
%\end{equation}
%\noindent
%In particular
%\begin{equation}
%{\rm{eval}}\{ \exp( \mathfrak{E}(x)) \} = 
%\mathbb{E} \left[ it( x - 1/2 + i L_{E}) \right]
%\end{equation}
%\noindent
%and 
%\begin{equation}
%%E_{n}(x) = \mathbb{E}\left[ ( x - \tfrac{1}{2} + i L_{E} )^{n} \right].
%\label{umbra-euler}
%\end{equation}
%\end{theorem}
\begin{proof}
The proof is similar to the Bernoulli case in Theorem \ref{density-B}. In this 
case, entry $3.981.3$ of \cite{gradshteyn-2007a}: 
\begin{equation}
\int_{0}^{\infty} {\rm{sech}}(ax) \cos(xt) \, dx = \frac{\pi}{2a} 
{\rm{sech}}\left( \frac{\pi t}{2a} \right)
\end{equation}
\noindent
is employed.
\end{proof}

\begin{note}
The analogue of Example \ref{umbra-psi} is 
\begin{equation}
{\rm{eval}} \left\{ \log \mathfrak{E}(x) \right\} = 
\log 2 + 2 \log \Gamma \left( \tfrac{x+1}{2} \right) -
2 \log \Gamma \left( \tfrac{x}{2} \right), \quad \text{ for } x > \tfrac{1}{2}
\end{equation}
\noindent
and differentiation produces 
\begin{equation}
{\rm{eval}}\left\{ \E^{-k}(x) \right \} = \frac{(-1)^{k-1}}{(k-1)!} 
2 \beta^{(k-1)}(x), \quad \text{ for } x > \tfrac{1}{2},
\label{beta-k}
\end{equation}
\noindent
with 
\begin{equation}
\beta(x) = \frac{1}{2} \left( \psi \left( \frac{x+1}{2} \right) - 
\psi \left( \frac{x}{2} \right)  \right)
\end{equation}
\noindent 
the beta function on page $906$ of \cite{gradshteyn-2007a}. 
\noindent
The proofs of all these results are similar to those presented for the 
Bernoulli case. 
\end{note}

It is natural to consider now the modified Euler numbers 
\begin{equation}
E_{n}^{*}=\sum_{r=0}^{n}\binom{n+r}{2r}\frac{n}{n+r}E_{r}, \quad n > 0.
\end{equation}
Symbolic experimentation suggested the next statement. The proof of the 
next statement follows the 
same ideas as in the Bernoulli case.

\begin{theorem}
The odd subsequence of the modified Euler numbers $\{E_{2n+1}^{*} \}$ is a 
periodic sequence of period $3$, with values  $\{ 1, \, -2, \, 1 \}$.
\end{theorem}

Define the modified Euler polynomials by 
\begin{equation}
E_{n}^{*}(x)=\sum_{r=0}^{n}\binom{n+r}{2r}\frac{n}{n+r}E_{r}(x).
\end{equation}
\noindent
Then the even-order numbers $E_{2n}^{*}(0)$ have period $12$ with values 

\begin{center}
\begin{tabular}{|c||c|c|c|c|c|c|}
\hline 
$n$ mod $12$ & 0 & 2 & 4 & 6 & 8 & 10\tabularnewline
\hline 
\hline 
$E_{n}^{*}(0)$ & 1 & 0 & -2 & 3 & -2 & 0\tabularnewline
\hline 
\end{tabular}
\par\end{center}

The proof of this result follows the same steps as in the Bernoulli case.

\medskip

The final statement in this section is the analogue of Theorem \ref{thm-symm}.

\begin{theorem}
The modified Euler polynomials  satisfy 
\begin{equation}
E_{n}^{*}(-x-3) = (-1)^{n} E_{n}^{*}(x).
\end{equation}
\end{theorem}

\section{The duplication formula for Zagier polynomials}
\setcounter{equation}{0}
\label{sec-duplication}

The identity 
\begin{equation}
B_{k}(mx) = m^{k-1} \sum_{k=0}^{m-1} B_{k} \left( x + \frac{k}{m} \right)
\end{equation}
\noindent
was given by J. L. Raabe in $1851$. The special case $m=2$ gives the
duplication formula for Bernoulli polynomials
\begin{equation}
2B_{k}(2x) = 2^{k} B_{k}(x) + 2^{k} B_{k} \left( x + \tfrac{1}{2} \right).
\label{raabe-2}
\end{equation}
\noindent
Summing over $k$ yields
\begin{equation}
2 \sum_{k=0}^{n} \binom{n+k}{2k} \frac{B_{k}(2x)}{n+k} = 
\sum_{k=0}^{n} \binom{n+k}{2k} \frac{2^{k}B_{k}(x)}{n+k} +
\sum_{k=0}^{n} \binom{n+k}{2k} \frac{2^{k}B_{k} \left(x + \tfrac{1}{2}
\right)}{n+k}.
\label{summing-1}
\end{equation}
\noindent
An umbral interpretation of this identity leads to a duplication 
formula for the Zagier polynomials. This result is expressed in terms 
of the umbral composition defined next.

\begin{definition}
Given two sequences of polynomials $P = \{ P_{n}(x) \}$ and 
$Q = \{ Q_{n}(x) \}$, their
\textit{umbral composition} is defined as 
\begin{equation}
(P \circ Q)_{n}(x) = \sum_{k=0}^{n} p_{k,n} Q_{k}(x),
\end{equation}
\noindent
where $p_{k,n}$ is the coefficient of $x^{k}$ in $P_{n}(x)$.
\end{definition}

The use of  umbral composition is clarified in the next lemma. 

\begin{lemma}
\label{comp-umbrae}
Let $P$ and $Q$ be polynomials and assume 
\begin{equation}
P_{n}(x) = {\rm{eval}} \left\{ \left( x + \pe \right)^{n} \right\} \text{ and }
Q_{n}(x) = {\rm{eval}} \left\{ \left(x + \Q \right)^{n} \right\}.
\end{equation}
\noindent
Then 
\begin{equation}
(P \circ Q)_{n}(x) =  {\rm{eval}} \left\{ \left(x + \pe + \Q \right)^{n} 
\right\}.
\end{equation}
\end{lemma}
\begin{proof}
Denoting the relevant umbrae by a subindex, then
\begin{eqnarray*}
{\rm{eval}}_{\pe, \Q} \left\{ \left( x + \pe + \Q \right)^{n} \right\} & = & 
{\rm{eval}}_{\Q} \left\{ P_{n}(x + \Q)  \right\} \\
& = & \sum_{k=0}^{n} p_{k,n}Q_{k}(x) \\
& = & (P \circ Q)_{n}(x),
\end{eqnarray*}
\noindent
as claimed.
\end{proof}

\medskip

Consider now the Bernoulli and Euler umbrae 
\begin{equation}
{\rm{eval}} \left\{ \exp( t \B ) \right\} = \frac{t}{e^{t}-1} 
\text{ and }
{\rm{eval}} \left\{ \exp( t \E ) \right\} = \frac{2}{e^{t}+1} 
\end{equation}
\noindent
given in \eqref{bn-umbrae1} and \eqref{umbra-eulerpoly}, respectively. The 
identity 
\begin{equation}
{\rm{eval}} \left\{ \exp( t \B ) \right\}  \times 
{\rm{eval}} \left\{ \exp( t \E ) \right\} = 
{\rm{eval}} \left\{ \exp( 2 t  \B ) \right\} 
\end{equation}
\noindent
is written (at the umbrae level) as 
\begin{equation}
\B + \E = 2 \B.
\end{equation}
\noindent 
The first summand on the right of  \eqref{summing-1} contains the term 
\begin{eqnarray*}
2^{k}B_{k}(x) & = & {\rm{eval}} \left\{ 2^{k} ( x + \B)^{k}  \right\} \\
  & = & {\rm{eval}} \left\{ ( 2 x + 2 \B )^{k} \right\} \\
  & = & {\rm{eval}} \left\{ ( 2 x + \B  + \E )^{k} \right\} \\
  & = & {\rm{eval}} \left\{ ( \B  \circ \E )_{k}(2x) \right\}.
\end{eqnarray*}
\noindent
Lemma \ref{comp-umbrae} has been used in the last step. Similarly 
\begin{equation*}
2^{k}B_{k} \left( x + \tfrac{1}{2} \right)  = 
  {\rm{eval}} \left\{ ( \B  \circ \E )_{k}(2x+1) \right\}.
\end{equation*}
\noindent
Thus, \eqref{summing-1} reads
\begin{equation}
2 B_{n}^{*}(2x) = (B^{*} \circ E)_{n}(2x) + (B^{*} \circ E)_{n}(2x+1)
\end{equation}
\noindent
that can also be expressed in the form 
\begin{equation}
2 B_{n}^{*}(2x) = (B^{*} \circ E(x))_{n}(x) + 
(B^{*} \circ E \left( x + \tfrac{1}{2} \right))_{n} \left( x + \tfrac{1}{2} 
\right),
\end{equation}
\noindent
that is an analogue of \eqref{raabe-2} for the Zagier polynomials.

\medskip

\no
\textbf{Acknowledgments}. The second author acknowledges the 
partial support of NSF-DMS 1112656. The first author is a post-doctoral 
fellow funded in part by the same grant.  The authors wish to thank 
T. Amdeberhan for his valuable input into this paper. 

%\bibliography{../../AllRef/official}

\begin{thebibliography}{10}

\bibitem{blissard-1861a}
J.~Blissard.
\newblock Theory of generic equations.
\newblock {\em Quart. J. Pure Appl. Math.}, 4:279--305, 1861.

\bibitem{boyadzhiev-2007a}
K.~Boyadzhiev.
\newblock A note on {B}ernoulli polynomials and solitons.
\newblock {\em Jour. Nonlinear Math. Phys.}, 14:174--178, 2007.

\bibitem{branden-2011a}
P.~Br\"and\'en.
\newblock Iterated sequences and the geometry of zeros.
\newblock {\em J. {R}eine {A}ngew. {M}ath.}, 658:115--131, 2011.

\bibitem{brychkov-2008a}
Y.~A. Brychkov.
\newblock {\em Handbook of {S}pecial {F}unctions. {D}erivatives, {I}ntegrals,
  {S}eries and {O}ther {F}ormulas}.
\newblock Taylor and {F}rancis, {B}oca {R}aton, {F}lorida, 2008.

\bibitem{dixit-2012b}
A.~Dixit, V.~Moll, and C.~Vignat.
\newblock The {Z}agier modification of {B}ernoulli numbers and a polynomial
  extension. {P}art {II}.
\newblock {\em Preprint}, 2012.

\bibitem{gessel-2003a}
I.~Gessel.
\newblock Applications of the classical umbral calculus.
\newblock {\em Algebra Universalis}, 49:397--434, 2003.

\bibitem{gradshteyn-2007a}
I.~S. Gradshteyn and I.~M. Ryzhik.
\newblock {\em Table of {I}ntegrals, {S}eries, and {P}roducts}.
\newblock Edited by A. Jeffrey and D. Zwillinger. Academic Press, New York, 7th
  edition, 2007.

\bibitem{grosset-2005a}
M.~P. Grosset and A.~P Veselov.
\newblock Bernoulli numbers and solitions.
\newblock {\em Jour. Nonlinear Math. Phys.}, 12:469--474, 2005.

\bibitem{kervaire-1963a}
M.~Kervaire and J.~Milnor.
\newblock Groups of homotopy spheres: I.
\newblock {\em Ann. Math.}, 77:504--537, 1963.

\bibitem{levine-1983a}
J.~P. Levine.
\newblock Lectures on groups of homotopy spheres.
\newblock In A.~Ranicki N. Levitt~F. Quinn, editor, {\em Alegebraic and
  {G}eometric {T}opology. Lecture {N}otes in {M}athematics, $1126$}, pages
  62--95. Springer, Berlin - Heidelberg - New York, 1983.

\bibitem{macmillan-2011a}
K.~Mac{M}illan and J.~Sondow.
\newblock Proofs of power sum and binomial coefficient congruences via
  {P}ascal's identity.
\newblock {\em Amer. Math. Monthly}, 118:549--551, 2011.

\bibitem{olver-2010a}
F.~W.~J. Olver, D.~W. Lozier, R.~F. Boisvert, and C.~W. Clark, editors.
\newblock {\em {NIST} {H}andbook of {M}athematical {F}unctions}.
\newblock Cambridge {U}niversity {P}ress, 2010.

\bibitem{petkovsek-1996a}
M.~Petkov\v{s}ek, H.~Wilf, and D.~Zeilberger.
\newblock {\em A=B}.
\newblock A. K. Peters, Ltd., 1st edition, 1996.

\bibitem{ribenboim-1999a}
P.~Ribenboim.
\newblock {\em Fermat's {L}ast {T}heorem for {A}mateurs}.
\newblock Springer-Verlag, New York, 1st edition, 1999.

\bibitem{riordan-1968a}
J.~Riordan.
\newblock {\em Combinatorial Identities}.
\newblock Wiley, New York, 1st edition, 1968.

\bibitem{spanier-1987a}
J.~Spanier and K.~Oldham.
\newblock {\em An atlas of functions}.
\newblock Hemisphere {P}ublishing Co., 1st edition, 1987.

\bibitem{stewarti-1979a}
I.~Stewart and D.~Tall.
\newblock {\em Algebraic {N}umber {T}heory}.
\newblock Chapman and {H}all, {L}ondon, 1st edition, 1979.

\bibitem{temme-1996a}
N.~M. Temme.
\newblock {\em Special {F}unctions. {A}n introduction to the {C}lassical
  {F}unctions of {M}athematical {P}hysics}.
\newblock John Wiley and sons, New York, 1996.

\bibitem{touchard-1956a}
J.~Touchard.
\newblock Nombres exponentieles et nombres de {B}ernoulli.
\newblock {\em Canad. J. {M}ath.}, 8:305--320, 1956.

\bibitem{yu-2012a}
Y.~P. Yu.
\newblock Bernoulli operator and {R}iemann's {Z}eta function.
\newblock {\em \text{ArXiv: math-NT}/1011.3352$\times$v3, 19541}, 2012.

\bibitem{zagier-1998a}
D.~Zagier.
\newblock A modified {B}ernoulli number.
\newblock {\em Nieuw {A}rchief voor {W}iskunde}, 16:63--72, 1998.

\end{thebibliography}
%\bibliographystyle{plain}

\end{document}